\documentclass[a4paper, reqno, 11pt]{amsart}

\usepackage[english]{babel}
\usepackage{amsmath}
\usepackage{amssymb}
\usepackage{amsthm}
\usepackage{enumerate}
\usepackage{ifthen}
\usepackage{bbm}
\usepackage{color}
\provideboolean{shownotes} 
\setboolean{shownotes}{true}
\usepackage{hyperref}
\usepackage{graphicx}
\usepackage{pgfplots}
\usepackage{tikz,tikz-3dplot} 
\usepackage{mathtools}
\usepackage{float}
\usepackage{geometry}
\newcommand{\margnote}[1]{
\ifthenelse{\boolean{shownotes}}%
{\marginpar{\raggedright\tiny\texttt{#1}}}%
{}%
}

\newcommand{\hole}[1]{
\ifthenelse{\boolean{shownotes}}%
{\begin{center} \fbox{ \rule {.25cm}{0cm}
\rule[-.1cm]{0cm}{.4cm} \parbox{.85\textwidth}{\begin{center}
\texttt{#1}\end{center}} \rule {.25cm}{0cm}}\end{center}}
{}
}
\usepackage{mathrsfs}
\newtheorem{thm}{Theorem}[section]

\newtheorem{lem}[thm]{Lemma}

\newtheorem{rem}[thm]{Remark}

\newtheorem{defn}[thm]{Definition}

\newcommand{\e}{\varepsilon}		       
\newcommand{\R}{\mathbb{R}}
\newcommand{\T}{\mathbb{T}}
\newcommand{\N}{\mathbb{N}}
\newcommand{\Z}{\mathbb{Z}}
\newcommand{\dive}{\mathop{\mathrm {div}}}
\newcommand{\curl}{\mathop{\mathrm {curl}}}

\newcommand{\de}{\mathrm{d}}

\numberwithin{equation}{section}

\subjclass[35A]{35Q35, 76E25, 76W05}


\begin{document}

\title[Magnetic reconnection for MHD]{Magnetic reconnection in Magnetohydrodynamics}

\author[P. Caro]{Pedro Caro}
\address[P.Caro]{BCAM - Basque Center for Applied Mathematics, Alameda de Mazarredo 14, E48009 Bilbao, Basque Country - Spain and Ikerbasque, Basque Foundation
for Science, 48011 Bilbao, Basque Country - Spain.}
\email[]{\href{pcaro@}{pcaro@bcamath.org}}

\author[G. Ciampa]{Gennaro Ciampa}
\address[G.\ Ciampa]{Dipartimento di Matematica "Federigo Enriques", Universit\`a degli Studi di Milano, Via Cesare Saldini 50, 20133 Milano, Italy.}
\email[]{\href{gciampa@}{gennaro.ciampa@unimi.it}}

\author[R. Luc\`a]{Renato Luc\`a}
\address[R. Luc\`a]{BCAM - Basque Center for Applied Mathematics, Alameda de Mazarredo 14, E48009 Bilbao, Basque Country - Spain and Ikerbasque, Basque Foundation
for Science, 48011 Bilbao, Basque Country - Spain.}
\email[]{\href{rluca@}{rluca@bcamath.org}}

\begin{abstract}
We provide examples of periodic solutions (in both 2 and 3 dimension) of the Magnetohydrodynamics equations such that the topology of the magnetic lines changes during the evolution. This phenomenon, known as magnetic reconnection, is 
relevant for physicists, in particular in the study of highly conducting plasmas. Although numerical and experimental evidences exist, analytical examples of magnetic reconnection were not known.
\end{abstract}

\maketitle

\section{Introduction}
We are interested in the Magnetohydrodynamics system, i.e.
\begin{equation}\label{eq:mhd}\tag{MHD}
\begin{cases}
\partial_t u+(u\cdot \nabla)u+\nabla P=\nu\Delta u+(b\cdot \nabla)b,\\
\partial_t b+(u\cdot \nabla)b=(b\cdot \nabla)u+\eta\Delta b,\\
\dive u=\dive b=0,\\
u(0,\cdot)=u_0,\hspace{0.3cm} b(0,\cdot)=b_0,
\end{cases}
\end{equation}
where (for $d=2,3$) $b:(0,T)\times\T^{d} \to \R^{d}$ identifies the magnetic field in a resistive incompressible 
fluid with velocity $u: (0,T)\times\T^{d} \to \R^{d}$. The scalar quantity $P:(0,T)\times\T^d \to \R$ is the 
 pressure, $\nu\geq 0$ is the viscosity and $\eta\geq 0$ is the resistivity. The system \eqref{eq:mhd} describes the behaviour of an electrically conducting incompressible fluid, the equations are given by a combination of the Navier--Stokes equations and Maxwell's equations from electromagnetisms.\\

In the resistive and viscous case (i.e. $\eta>0,\nu>0$), the existence of global weak solutions with finite energy and local strong solutions to \eqref{eq:mhd} in two and three dimensions have been proved in~\cite{DL72}. Moreover, for smooth initial data they proved the smoothness and uniqueness of their global weak solutions in the two-dimensional case. On the other hand, in \cite{ST} the authors proved the uniqueness of the local strong solutions in 3D, together with some regularity criteria. This situation is somewhat reminiscent of the available results for the Navier--Stokes equations. A similar situation arises in the non-viscous case, i.e. when $\nu=0$: In 2D the existence 
of global weak solutions has been proved in \cite{K} for divergence-free initial data in $L^2$. In 3D, similarly to the viscous case, local smooth solutions exist and are unique; moreover, if $\nabla u\in L^1L^\infty$ the solution is also global. 
The ideal case $\eta =0$ has attracted the attention of many mathematicians in recent years and local
well-posedness results, at an (essentially) sharp level of Sobolev regularity, are now available 
\cite{FMRR, Feff}. See also \cite{LXZ, LZ, RWXZ, XZ}, for global existence results of smooth solutions in the ideal case assuming that the initial datum is a small perturbation of a constant steady state.

We are interested in the problem of \emph{magnetic reconnection}. In the non-resistive case ($\eta=0$) it is known that the integral lines of a sufficiently smooth magnetic field are transported by the fluid ({\em Alfven's theorem}). In particular 
the topology of the integral lines of the magnetic field does not change under the evolution. 
The topological stability of the magnetic structure is related to the conservation of the magnetic helicity, which, in the non resistive case ($\eta = 0$), becomes a very subtle matter at low regularities, intimately related to anomalous dissipation phenomena.
We refer to \cite{BBV20,FL19,FLS20}  for some very interesting (positive and negative) results in this direction.
On the other hand, in the resistive case ($\eta > 0$) the topology of the magnetic lines 
may (and it is indeed expected to) change under the fluid evolution, in both 2 and 3 dimensions, even for regular solutions. This phenomenon, known as magnetic reconnection, is of particular relevance for physicists, in particular in the study of highly conducting plasmas.
A possible explanation of the phenomenon of the solar flares, large releases of energy from the surface of the sun, involves magnetic reconnection. The energy stored in the magnetic fields over a large period of time is rapidly released during the change of topology of the magnetic lines. It is also worth mentioning that the intensity of the solar flares is of a larger magnitude than the one predicted by the current \eqref{eq:mhd} models, suggesting that also some turbulent phenomenon, as cascade of energy, may be involved \cite{Priest}. Although numerical and experimental evidences exist (see \cite{MagnPhys, Priest} and the references therein), no analytical examples of magnetic reconnection are known.
Besides the intrinsic mathematical interest, a better understanding by a rigorous analytic viewpoint 
may give important insights on the Sweet-Parker model arising in magnetic reconnection theory \cite{Priest}.
\\
\\
Thus, in this work we are concerned with providing analytical examples of this phenomenon. Our main result is the following.
\begin{thm}\label{thm:main}
Consider $d\in\{2,3\}$. Given any viscosity and resistivity $\nu, \eta >0$ and any constants $T>0$ and $M>0$ there exists a zero-average unique global smooth solution $(u,b)$ of \eqref{eq:mhd} on $\T^d$, with initial datum $(0, b_0)$ and $\|b_0\|_{L^2}=M$, such that the magnetic lines at time $t=0$ and $t=T$ are not topologically equivalent, meaning that there is no homeomorphism of $\T^d$ into itself mapping the 
magnetic lines of $b(0,\cdot)$ into that of $b(T,\cdot)$.
\end{thm}

Note that as $M$ may be very large and the solutions have zero-average, we are considering genuinely large initial magnetic fields (for instance large in any Sobolev space). It is also possible to consider more general initial velocities $u_0 \neq 0$. In particular we may consider large initial  
velocity if we impose some a priori structure (see Remark \ref{Remark:LargeDataWithStructure1}), 
however we prefer to state the result in the simplest form. Notably, large velocities (without any specific geometrical structure) may be also considered in the 2D case 
if we work with very large viscosities (see Theorem \ref{thm:main2}).

The proofs of Theorem \ref{thm:main} in 2D and in 3D are logically independent. For the sake of readability, we first present the proof of the theorem in 3D, which is and adaptation of the argument from~\cite{ELP}. The proof of the result in 2D requires new ideas and it is somehow more general, as it should be clear by the following observation: if we have found a 2D solution $$(b_1(t,x_1, x_2), b_2(t,x_1, x_2))$$ of \eqref{eq:mhd}  which exhibits a reconnection, then we have  also proved magnetic reconnection for \eqref{eq:mhd} in 3D, simply considering the 3D solution 
\begin{equation}\label{AsInHere}
(b_1(t,x_1, x_2), b_2(t,x_1, x_2), 0);
\end{equation}
the 2D velocity and pressure must be extended to 3D in the analogous way.
 On the other hand, an advantage of the genuinely 3D argument will be the possibility to (additionally) prescribe rich 
 topological structures for the magnetic lines, relying upon some deep results about topological richness of Beltrami fields \cite{Annals, Acta, EPT}. Moreover, the 3D reconnection obtained extending the 2D result to 3D as in \eqref{AsInHere} would not be
 {\it structurally stable} in the sense of the Remark \ref{Rem:StructStable} below. We explain why in
 Remark \ref{Rem:FailsSStab}.

It is worth mentioning already that the 2D result and the (genuine) 3D result that we presented are 
structurally stable indeed; see again Remark \ref{Rem:StructStable}.

As we have noted, the condition $\eta >0$ is necessary to prove magnetic reconnection for smooth solutions, 
because for $\eta=0$ this is forbidden by Alfven's theorem. The heuristic behind the phenomenon is 
that the resistivity allows to break the topological rigidity. In this sense, it is interesting that we can prove magnetic reconnection at arbitrarily small resistivity, namely for all $\eta >0$, thus even in a very turbulent regime. 

Regarding the role played by the viscosity, we mention that Theorem \ref{thm:main} might be extended to the 
case $\nu =0$, working with growth and stability estimates like the ones used in the Euler equations theory, rather than in the Navier--Stokes one. The price that we must pay to work with 
viscosity~$\nu =0$ is of course that we will only have local in time results. 
%

However, the viscosity may enter in the argument in an interesting way, since it is valid the 
principle that large viscosity helps too. We will investigate this in the 2D case, where,
 we can prove a stronger reconnection statement which is valid for any initial velocity 
in~$H^4(\T^2)$ as long as we work with sufficiently 
large viscosity.

\begin{thm}\label{thm:main2}
Consider $d=2$. Given any resistivity $\eta >0$ and any constants $T>0$ and $M, R>0$, there exists a viscosity
$\nu = \nu(M, R) \gg 1$ sufficiently large such that the following holds: For any zero-average 
$u_0$ with $\| u_0 \|_{H^4} = R$, there exists a 
zero-average unique global smooth solution $(u,b)$ of \eqref{eq:mhd} on $\T^2$, with initial datum $(u_0, b_0)$ 
and $\|b_0\|_{L^2}=M$, such that the magnetic lines at time $t=0$ and $t=T$ are not topologically equivalent.
\end{thm}

\begin{rem}\label{Rem:StructStable}
The results that we presented above are {\it structurally stable}, in the sense that if we slightly perturb the initial data 
(in the appropriate norms) and/or the observations times~$t = 0$,~$t=T$, the results are still valid.  This will be clear 
by the proof (see also Remark \ref{Remark:LargeDataWithStructure1}). 
The structural stability of the 
phenomenon is important from a physical point of view, since it guarantees its observability. 
\end{rem}

The proofs are based on a perturbative analysis of some particular solutions of the linearized equation, for which one can infer reconnection by a suitable topological argument. More precisely, in the 3D case 
we closely follow the idea of \cite{ELP}, considering data for which the magnetic field at time $t=0$ has the form: 
$$
M B_0 + \delta B_1, \qquad M>0, \quad 0<  \delta \ll 1,
$$ 
where $B_j$ are high frequency Beltrami fields (eigenvector of the curl operator) with the following properties:
\begin{enumerate}
[$i)_{3D}$]
\item All the magnetic lines of $B_0$ wing around a certain direction of the torus, in particular they are all non contractible. This 
is a robust topological property, in the sense that it is still valid for all sufficiently small regular perturbations of $B_0$.
\item The field $B_1$ has some contractible magnetic line (in fact in a small ball we can prescribe magnetic lines knotted and linked in complicated ways following the topological results from \cite{ELP, Annals, Acta, EPT}). This is again topologically robust.
\end{enumerate}
The idea is then to choose the relevant parameters, namely $\delta$ and the eigenvalues (frequencies)
of the Beltrami fields in such a way that our solution will be sufficiently close (in a regular norm) to $MB_0$ at time 
$t=0$ and to a suitable rescaled version of the field $B_1$, at time $T >0$. 
Recalling the topological constraint $i)_{3D}, ii)_{3D}$ (and their robustness), this will ensure that the 
magnetic lines of the solution at time $t=0$ and $t=T$ are not homeomorphic. Thus 
we must have had magnetic reconnections in the intermediate times. 

If we try to use the same strategy in 2D we encounter the problem that it is not easy to produce 
a simple high frequency Taylor vector field (which is the 2D analogous of a Beltrami field) 
 with all non contractible vortex lines.
Thus,
in the 2D case we use a different topological constraint, that consists in counting the number of stagnation 
points of the magnetic field. In particular we  
define the initial magnetic field as
$$
M V_N + \delta V_1, \qquad M > 0, \quad 0 <  \delta \ll 1,
$$ 
where $V_N$ and $V_1$ are Taylor fields 
(eigenvectors of the Stokes problem \eqref{eq:eigen-stokes}) with the 
following properties:
\begin{enumerate}[$i)_{2D}$]
\item The field $V_N$ has several stagnation 
points (namely $\sim N \gg 1$, half of them are hyperbolic and half of them elliptic). 
This is robust in the sense that small regular perturbations must have at least as many stagnation points.
\item The fields $V_1$ has exactly four stagnation 
points, and the topology of the vortex lines is completely prescribed (see figure \ref{fig:v1}) and 
robust ($V_1$ is 
structurally stable by Theorem \ref{thm:MW}).
\end{enumerate}
Again, one will then choose the relevant parameters in such a way that the solution will share the same 
topological properties $i)_{2D}$ and $ii)_{2D}$ at times $t=0$ and $t=T > 0$, respectively. 
This proves magnetic reconnection for intermediate times.

Finally, in the last part of the paper we provide an example of initial magnetic fields which shows instantaneous reconnection under the \eqref{eq:mhd} flow. The theorem is the following:
\begin{thm}\label{thm:main3}
Consider $M>0$ and $u_0 \in C^{\infty}(\T^d)$ with zero-average, where $d \in \{2, 3\}$. There exists a zero-average
initial 
magnetic field  $b_0$ with $\| b_0 \|_{L^2} = M$ such that the following holds: If $d=3$
there is a zero-average local smooth solution of \eqref{eq:mhd} with initial datum $(u_0, b_0)$ 
which at time $t=0$ has a tube of magnetic lines that breaks up instantaneously (namely for any positive time $t >0$).
If $d=2$
there is a zero-average global smooth solution with initial datum $(u_0, b_0)$ 
which at time $t=0$ has an heteroclinic connection that breaks up instantaneously.   
\end{thm}

The 3D case of the theorem above closely follows the idea of \cite{ELP}: by considering an initial datum which is 
not structurally stable (again the datum will be a small perturbation of a Beltrami field), one can prove that some 
magnetic lines sitting on a resonant (embedded 2D) torus rearrange 
instantaneously their topology. In the 2D case we exploit the structural instability of the heteroclinic connections 
of a suitable perturbation of a Taylor field. In particular, we show that an heteroclinic connection is instantly broken. 
Both results may be proved invoking the Melnikov theory, however in 2D a significant 
simplification of the argument is available (see Section \ref{Sec:Instant}). We are grateful to Daniel Peralta-Salas for this observation. 

\subsection{Organization of the paper}
In the rest of the introduction we describe the notation and some general facts that we will be using throughout the article. In  the sections \ref{Sec:BeltramiTaylor} and \ref{Sec:BeltramiTaylorBis}, we recall the concepts of Beltrami and Taylor fields that will be necessary in the proof of the magnetic reconnection for 3D and 2D, respectively. One basic ingredient in our proofs is the stability of {\it strong} solutions of the MHD system. The sections \ref{Sec:3dStab} and \ref{Sec:2dStab} are devoted to this matter in the 3D and 2D case, respectively. The magnetic reconnection in the 3D case is proved in the section \ref{Sec:Proof3D}. The magnetic reconnection in the 2D case is proved in the sections 
\ref{Sec:SmallVel} in the case of small velocities and in section \ref{Sec:Proof2DLarge} in the case of large velocities, but under an additional assumption on the size of the viscosity. 
Finally, the section \ref{Sec:Instant} is devoted to the instantaneous reconnection.

%
%

\subsection{Notations and preliminaries} Throughout the paper, we will denote by $C$ a positive constant whose value can change line by line. We will denote by $\T^d := \R^d/\Z^d$ the $d$-dimensional flat torus equipped with 
the Lebesgue measure $\mathscr{L}^d$. We will always work with 
$d\in \{2,3\}$. 
For a given positive integer $m$ and a given $d$-dimensional vector field $w:\T^d\to\R^d$, we define
$$
|\nabla^m w|^2:= \sum_{|\alpha|=m} |\partial^\alpha w|^2,
$$
where $\alpha \in \N^d$ is a multi-index, and
$$
\|w\|^2_{H^r}:=\sum_{m=0}^r\int_{\T^d}|\nabla^m w(x)|^2 \de x,
$$
where $r$ is a given positive integer. We will use $p,q$ to denote real numbers in $[1,+\infty]$. 
We will adopt the customary notation for Lebesgue spaces $L^p(\T^d)$ and for Sobolev spaces $W^{k,p}(\T^d)$; in particular, $H^k(\T^d) := W^{k,2}(\T^d)$. We will denote with $\|\cdot\|_{L^p}$ (respectively $\|\cdot\|_{W^{k,p}}$,$\|\cdot\|_{H^k}$) the norms of the aforementioned functional spaces, omitting the domain dependence. Every definition below can be adapted in a standard way to the case of spaces involving time, like e.g. $L^1([0,T];L^p(\T^d))$. Moreover, for a time-dependent vector field $w(t,x)$, we define
$$
 \|w\|^2_{L^2W^{r,\infty}}:=\sum_{m=0}^r\int_0^\infty\|\nabla^m w(t,\cdot)\|_{\infty}^2 \de x.
$$

Working in the 2D case, we will frequently use the interpolation inequality 
\begin{align*}
\| f \|_{L^4(\T^2)} \leq  C 
\| f \|^{1/2}_{L^2(\T^2)} \| \nabla f \|^{1/2}_{L^2(\T^2)}  
\end{align*}
valid for zero-average functions, to which we refer as Ladyzenskaya's inequality.

In 3D a similar role will be played by 
\begin{equation*}
\| f \|_{L^\infty(\T^3)} \leq C 
\|  f \|^{1/2}_{L^6(\T^3)}
\|\nabla f \|^{1/2}_{L^2(\T^3)},
\end{equation*}
again valid for zero-average functions, to which we refer as Gagliardo-Nieremberg's inequality inequality. Finally, we recall the well-known Gronwall's lemma.

\begin{lem}[Gronwall]
Let $f$ be a nonnegative, absolutely continuous function on $[0,T]$, which satisfies for a.e. $t$ the differential inequality
$$
f'(t)\leq \alpha(t)f(t)+\beta(t),
$$
where $\alpha,\beta$ are nonnegative, summable functions on $[0,T]$. Then
$$
f(t)\leq e^{A(t)} \left(f(0)+\int_0^t\beta(s)e^{-A(s)}\de s\right),
$$
for all $t\in[0,T]$, where $A(t)=\int_0^t \alpha(s)\de s$.
\end{lem}

\section{Beltrami fields}\label{Sec:BeltramiTaylor}

In this section we introduce the main mathematical objects that we need for the 3D magnetic reconnection result. 
These are the so-called Beltrami fields.

A vector field $B : \T^3\to\R^3$ is called a {\em Beltrami field} with frequency $N$ if it is an eigenfunction of the $\curl$ operator with eigenvalue $N \in \Z$, i.e.
\begin{equation}\label{eq:LambdaBeltrami}
\curl B =N B.
\end{equation}
It is in fact easy to check that on $\T^3$ the eigenvalues are have the form $|k|$ where $k \in \mathbb{Z}^3$.
We will restrict our attention to Beltrami fields of non-zero frequency, which are necessarily divergence-free and have zero mean, i.e.
$$
\dive B=0, \hspace{0.4cm}\int_{\T^3} B \, \de x = 0, \qquad \mbox{if $N \neq 0$}.
$$
The general form of a Beltrami field of frequency~$N$ is indeed 
\[
W=\sum_{|k|=\pm N}\Big(b_k \, \cos(k\cdot x)+\frac{b_k\times
  k} N\,\sin(k\cdot x)\Big)\,,
\]
where $b_k\in\R^3$ are vectors orthogonal to~$k$: $k\cdot b_{k}=0$. It is easy to check that $B$ satisfies additionally
\begin{align}
(B\cdot\nabla)B=\nabla\frac{|B|^2}{2},\\
\Delta B=-N^2B.
\end{align}
For our purposes we consider two topologically non equivalent Beltrami fields. The first is given explicitly  
\begin{align}
B_{0}:= (2\pi)^{-3/2}(\sin (N_0 x_3), \cos (N_0 x_3),0).\label{b0}
\end{align}
Note that all the integral lines of $B_0$ (which coincide with that of $\nabla \times B_{0}$) are either periodic or quasi-periodic (depending on the rationality/irrationality of the ratio $\frac{\sin (N_0 x_3)}{\cos (N_0 x_3)}$) and wind around the 2D tori given by the equation $x_3 = const$. In particular, all the integral lines of $B_{0}$ are non-contractible. This property is structurally stable, in the sense that it is still true for small (regular enough) perturbations.  More precisely, if   
\begin{equation}\label{quantstab2}
\| B_{0} -   B'\|_{C^{3,\alpha}}<\eta
\end{equation}
with a small ($N_0$-independent) constant, then 
\begin{enumerate}[(i)]
\item $B'$ does not have any contractible integral line (since the same is true for $B_{0}$).
\end{enumerate}
This fact, which is a KAM type theorem, was proved in Lemma 4.2 of \cite{ELP}.

The second Beltrami field 
is constructed in the following theorem, that is a simplified version of
Theorem 2.1 in \cite{ELP}. It is worth mentioning that 
the most important part of the proof of this result comes from \cite{Annals, Acta, EPT}. 

\begin{thm}[see Theorem 2.1 in \cite{ELP}]\label{Thmdef:BN1}
Let $S$ be a finite union of closed curves (disjoint, but possibly knotted and linked) in $\mathbb{T}^3$  that is contained in the unit ball. For all $N_1$ large enough and odd there exists a Beltrami field $B_1$  with some integral lines diffeomorphic to $S$ (namely related by a diffeomorphism of $\mathbb{T}^3$). This set is contained in a ball of radius $1/N_1$ and structurally stable, namely there exists  $\eta$ independent of $N_1$ such that any $B''$  such that 
\begin{equation}\label{quantstab}
\| B_{1} -   B'' \|_{C^{1}}<\eta
\end{equation}
has a collection of integral lines diffeomorphic to $S$.
Moreover
\begin{equation}\label{B-norm}
\frac1{C N_1}< \| B_{1}\|_{L^2}< \frac C{\sqrt N_1}.
\end{equation}
\end{thm}

In particular    
\begin{enumerate}[(ii)]
\item $B''$ has some contractible integral line (since the same is true for $B_{1}$). 
\end{enumerate}

In fact the theory developed in \cite{Annals, Acta, EPT} allows also to prescribe some vortex tubes 
of arbitrarily complicated topology
which can be realized by the vortex lines of the Beltrami field $B_1$. In this case the structural stability requires 
a stronger norm, namely $C^{3, \alpha}$. For the purpose of this paper we will consider the simple scenario 
from Theorem \ref{Thmdef:BN1}, as it is sufficient to prove magnetic reconnection.

The natural numbers $N_0,N_1$ will play the role of free parameters in our construction but eventually we will chose one of them much larger than the other in such a way that our solution at time $t=0$ will be close to (a rescaled version of) $B_{0}$ while at time $t = T >0$ will be close to (a rescaled version of) $B_{1}$. Thus, as consequence of $(i), (ii)$ above, a change of topology of the integral lines must have happened. 

%

\section{Stability of regular solutions of the MHD system in 3D}\label{Sec:3dStab}

The goal of this section is to provide a stability result for regular solutions of the Magnetohydrodynamic system. Later we will often work with some special reference solutions with zero velocity, however at this stage we prefer to prove a slightly more general perturbative result. The estimates below will be crucial for two reasons: on the one hand, they will quantify the error in the perturbative argument (see the remark below); on the other hand, they will allow us to construct a global solution as a small perturbation of some large global smooth solution (that will indeed be a Beltrami field).
 
%
%
%

The possibility to run a perturbative argument around strong solutions is of course not a novelty, however the important aspect of the next proposition is that we quantify the error in such a way that it depends only polynomially by the $L^2((0,T);W^{r,\infty}(\T^3))$ norm of the reference solution $(u, b)$, for $r \geq 1$, while the exponential dependence only involves its $L^2((0,T);L^{\infty}(\T^3))$ norm; see \eqref{est:stab}. This will be important to handle large initial data in our main theorem.

The initial data will be small perturbations of $(u_0,b_0)$, namely we focus on divergence-free vector fields $(w_0,m_0)$ such that
\begin{equation}
\|u_0-w_0\|_{H^r}+\|b_0-m_0\|_{H^r} \ll 1.
\end{equation}
We want to show that there exists a unique global regular solution $(w,m)$ starting from $(0,m_0)$. To do that, we proceed as in \cite{ELP}: we know that there exists a local solution $(w,m)$ and we prove that it can be extended globally, under suitable estimates for the Sobolev norms.

\begin{thm}\label{thm:stab}
Given some integer $r\geq 1$ and any $\sigma < \min ( \eta, \nu )$, let $(u,b) \in L^2((0,T);W^{r,\infty}(\T^3))$ be a global smooth solution of $\eqref{eq:mhd}$ with initial datum $(u_0,b_0)$ of zero mean such that 
\begin{equation}\label{AssOnWInfty}
\|u\|_{L^2 W^{s,\infty}}+\|b\|_{L^2W^{s,\infty}}<C(1+N^{s-1})
\end{equation}
for all integers $0 \leq s \leq r$, where $N \geq 1$. 
Then, there exists a sufficiently large positive constant $c$ such that, for any divergence-free vector field $(w_0,m_0)$ with zero mean and
\begin{equation}\label{assOnPertData3d}
\|u_0-w_0\|_{H^r}+\|b_0-m_0\|_{H^r} \leq \frac{1}{c} N^{1-r},
\end{equation}
the corresponding solution $(w,m)$ to \eqref{eq:mhd} is global and satisfies
\begin{align}\label{est:stab}
\|u(t,\cdot)-w(t,\cdot)\|_{H^s}+ & \|b(t,\cdot)-m(t,\cdot)\|_{H^s} 
\\ \nonumber
& \leq C (1 + N^{s-1}) e^{C(\int_0^t ( \|u (\cdot, s)\|^2_{L^\infty}  + \|b (\cdot, s)\|^2_{L^\infty} )ds} (\|u_0-w_0\|_{H^s}+\|b_0-m_0\|_{H^s})e^{- \sigma t},
\end{align}
for all $0\leq s\leq r$ and all $t>0$, with a $\sigma$-dependent constant $C$.
\end{thm}

\begin{rem}\label{RemarkAfterCor}
In the next Section we will apply this theorem with the choice $(u, b) = (0, M e^{-\eta N_0^2 t} B_0)$, where $B_0$ is the Beltrami field defined in \ref{b0}. It is immediate to check that this solves \eqref{eq:mhd} with an appropriate choice of the pressure (see next section) and that the assumptions of the theorem are satisfied with $N=N_0$.
\end{rem}

\begin{proof}
We denote by $P_{u,b}$ and $P_{w,m}$ the pressure function of, respectively, $(u, b)$ and $(w, m)$. We know from \cite{ST} that there exists a unique local solution $(w,m)$ of \eqref{eq:mhd} starting from $(w_0,m_0)$, and denote by $T^*$ the local time of existence. We start by proving the bound \eqref{est:stab} which will be enough to guarantee that the solution is actually global.\\
\\
\underline{\em Step 1} \hspace{0.5cm}Preliminaries.\\
\\
We define $v=w-u$ and $h=m-b$ it is easy to check that $(v,h)$ solves the following system
\begin{equation}\label{eq:difference}
\begin{cases}
\partial_t v+\dive\left(v\otimes v+2v\otimes u\right)+\nabla P_{v,h}=\nu\Delta v+\dive\left(h\otimes h+2h\otimes b\right),\\
\partial_t h+(v\cdot\nabla )h+(u\cdot \nabla)h+(v\cdot\nabla)b=\eta\Delta h+(h\cdot\nabla)v+(h\cdot\nabla)u+(b\cdot\nabla)v, \\
\dive v=\dive h=0,\\
v_0=w_0-u_0,\hspace{0.3cm}h_0=m_0-b_0,
\end{cases}
\end{equation}
where $P_{v,h}=P_{w, m}-P_{u,b}$. We recall that, since $(u,b)$ is a global smooth solution, the following energy equalities hold
\begin{equation}
\frac{1}{2}\frac{\de}{\de t}\int_{\T^3}\left(|u(t,x)|^2+|b(t,x)|^2\right)\de x=-\nu\int_{\T^3}|\nabla u(t,x)|^2\de x-\eta \int_{\T^3}|\nabla b(t,x)|^2\de x,
\end{equation}
\begin{equation}
\frac{1}{2}\frac{\de}{\de t}\int_{\T^3}|u(t,x)|^2\de x=-\nu\int_{\T^3}|\nabla u(t,x)|^2\de x+\int_{\T^3}(b\cdot\nabla)b\cdot u\,\de x,
\end{equation}
\begin{equation}
\frac{1}{2}\frac{\de}{\de t}\int_{\T^3}|b(t,x)|^2\de x=-\eta\int_{\T^3}|\nabla b(t,x)|^2\de x+\int_{\T^3}(b\cdot\nabla)u\cdot b\,\de x.
\end{equation}
Let us now define the time energies $e_s$ as follows
\begin{equation}
e_s(t):=\sum_{j=0}^s\int_{\T^3} \left(|\nabla^j v(t,x)|^2+|\nabla^j h(t,x) |^2\right)\de x.
\end{equation}
The goal now is to provide bounds on $e_s$ via an induction argument.\\
\\
\underline{\em Step 2} \hspace{0.5cm}Estimate on $e_0$.\\
\\
We multiply the first equation in \eqref{eq:difference} by $v$ and we obtain that
\begin{equation}
\partial_t\frac{|v|^2}{2}+(v\cdot\nabla)\frac{|v|^2}{2}+2(v\cdot\nabla)u\cdot v+\dive(P_{v,h} v)=\nu\Delta v\cdot v+(h\cdot\nabla )h\cdot v+ (h\cdot\nabla)b\cdot v + (b\cdot\nabla)h\cdot v,
\end{equation}
and integrating over $\T^3$ we get
\begin{equation}\label{eq:balance_v}
\frac{1}{2}\frac{\de}{\de t}\int_{\T^3}|v|^2\de x+2\int_{\T^3} (v\cdot\nabla)u\cdot v\,\de x=-\nu\int_{\T^3} |\nabla v|^2\de x+\int_{\T^3}(h\cdot\nabla )h\cdot v\de x+ \int_{\T^3}(h\cdot\nabla)b\cdot v\de x + \int_{\T^3}(b\cdot\nabla)h\cdot v\de x.
\end{equation}
We multiply the second equation in \eqref{eq:difference} by $h$ and, after integrating over $\T^3$ we get that
\begin{equation}\label{eq:balance_h}
\frac{1}{2}\frac{\de}{\de t}\int_{\T^3}|h|^2\de x+\int_{\T^3} (v\cdot\nabla)b\cdot h\,\de x=-\eta\int_{\T^3} |\nabla h|^2\de x+\int_{\T^3}(h\cdot\nabla )v\cdot h\de x+\int_{\T^3}(h\cdot\nabla)u\cdot h\de x+\int_{\T^3}(b\cdot\nabla)v\cdot h\de x.
\end{equation}
We use the identities
$$
\int_{\T^3}(h\cdot \nabla)h\cdot v\,\de x=-\int_{\T^3}(h\cdot\nabla)v\cdot h\,\de x,
$$
$$
\int_{\T^3}(b\cdot\nabla)h\cdot v\de x = - \int_{\T^3}(b\cdot\nabla)v\cdot h\de x
$$
and summing up \eqref{eq:balance_v} and \eqref{eq:balance_h} we get that
\begin{align}\label{eq:balance_vh}
\frac{1}{2}\frac{\de}{\de t}\int_{\T^3}\big(|v|^2+|h|^2\big)\de x&+2\int_{\T^3} (v\cdot\nabla)u\cdot v\,\de x+\int_{\T^3} (v\cdot\nabla)b\cdot h\,\de x=-\nu\int_{\T^3} |\nabla v|^2\de x-\eta \int_{\T^3}|\nabla h|^2\de x
\nonumber\\
&+\int_{\T^3}(h\cdot\nabla )u\cdot h\de x +\int_{\T^3}(h\cdot\nabla)b\cdot v\de x. 
\end{align}
By using integration by part and Young's inequality, we can obtain the following estimates
\begin{align*}
\left|\int_{\T^3} (v\cdot\nabla)u\cdot v\,\de x\right|&=\left|\displaystyle-\int_{\T^3} (v\cdot\nabla)v\cdot u\,\de x\right|\\
&\leq \|u\|_{L^\infty}\|v\|_{L^2}\|\nabla v\|_{L^2}\\
&\leq C\|u\|^2_{L^\infty}\|v\|_{L^2}^2+\varepsilon\|\nabla v\|_{L^2}^2,
\end{align*}
where $\varepsilon$ is a small constant that will be chosen later. Similarly 
\begin{itemize}
\item $\left| \displaystyle\int_{\T^3}(h\cdot\nabla)u\cdot h\de x\right|\leq C\|u\|_{L^\infty}^2\|h\|_{L^2}^2+\varepsilon\|\nabla h\|_{L^2}^2$
\vspace{0.2cm}
\item $\left| \displaystyle\int_{\T^3} (v\cdot\nabla)b\cdot h\,\de x\right|\leq C\|b\|^2_{L^\infty}\|v\|^2_{L^2}+\varepsilon\|\nabla h\|_{L^2}^2$
\vspace{0.2cm}
\item $\left| \displaystyle\int_{\T^3}(h\cdot\nabla)b\cdot v\,\de x\right|\leq C\|b\|^2_{L^\infty}\|h\|^2_{L^2}+\varepsilon\|\nabla v\|_{L^2}^2$
\end{itemize}
Then, by substituting in \eqref{eq:balance_vh}, we obtain that
\begin{equation*}
\frac{\de}{\de t} e_0(t)\leq C\left(\|u(t)\|_{L^\infty}^2+\|b(t)\|_{L^\infty}^2 \right)e_0(t)-2(\nu-2\varepsilon)\int_{\T^3} |\nabla v|^2\de x -2(\eta-2\varepsilon)\int_{\T^3}|\nabla h|^2\de x.
\end{equation*}
Finally, since $v$ and $h$ have zero mean for all times in which they are defined, we can use of Poincar\'e's inequality
$$
\|f\|_{L^2}\leq \|\nabla f\|_{L^2},
$$
and by properly fixing $\varepsilon$ we obtain that
\begin{equation*}
\frac{\de}{\de t} e_0(t)\leq\left[C\left(\|u(t)\|_{L^\infty}^2+\|b(t)\|_{L^\infty}^2 \right)-2\sigma\right]e_0(t),
\end{equation*}
where $\sigma\in(0,1)$ is a fixed quantity which depend on $\nu,\eta$. By Gronwall's lemma it follows that
\begin{equation}
e_0(t)\leq \left(\|v_0\|_{L^2}^2+\|h_0\|_{L^2}^2\right)\exp\left(C\int_0^t\|u(\tau)\|_{L^\infty}^2+\|b(\tau)\|_{L^\infty}^2\de \tau -2\sigma t \right),
\end{equation}
which leads to
\begin{align}
\|u(t,\cdot)-w(t,\cdot)\|_{L^2}^2 & +
\|b(t,\cdot)-m(t,\cdot)\|_{L^2}^2  
\\ \nonumber
& \leq \exp\left(C\int_0^t\|u(\tau)\|_{L^\infty}^2+\|b(\tau)\|_{L^\infty}^2\de \tau -2\sigma t \right) \left(\|u_0-w_0\|_{L^2}^2+\|b_0-m_0\|_{L^2}^2\right),
\end{align}
which implies  \eqref{est:stab} for $s=0$.\\
\\
\underline{\em Step 3} \hspace{0.5cm} Inductive step.\\
\\
We are assuming that 
%
$$
e_s(t)\leq C (1 + N^{2s-2}) \exp \left(C\int_0^t (\| u(\tau) \|_{L^\infty}^2+\| b(\tau) \|_{L^\infty}^2)\de \tau\right) e^{-2\sigma t}\left( \|v_0\|_{H^s}^2+\|b_0\|_{H^s}^2\right),
$$
for all $s<r$. We will show that the bound holds also for $s=r$. Let $\alpha\in\N^3$ with $|\alpha|\leq r$ and differentiate the equation for the velocity by $\nabla^\alpha$ to obtain
\begin{align}
\partial_t\partial^\alpha v-\nu\Delta\partial^\alpha v&+\nabla\partial^\alpha P_v+\sum_{\beta\leq \alpha}{\alpha \choose \beta}(\partial^\beta v\cdot\nabla)\partial^{\alpha-\beta}v+\sum_{\beta\leq \alpha}{\alpha\choose\beta}\left[(\partial^\beta u\cdot\nabla)\partial^{\alpha-\beta}v+(\partial^\beta v\cdot\nabla)\partial^{\alpha-\beta}u \right]\nonumber\\
&=\sum_{\beta\leq \alpha}{\alpha\choose\beta}\left[(\partial^\beta h\cdot\nabla)\partial^{\alpha-\beta}h+(\partial^\beta b\cdot\nabla)\partial^{\alpha-\beta}h+ (\partial^\beta h\cdot\nabla)\partial^{\alpha-\beta}b\right].
\end{align}
Multiply the above equation by $\partial^\alpha v$ and integrating in space we get
\begin{align*}
\frac{1}{2}\frac{\de}{\de t}&\int_{\T^3}|\partial^\alpha v|^2\de x+\nu\int_{\T^3}|\nabla\partial^\alpha v|^2\de x+\int_{\T^3}\sum_{\beta\leq \alpha}{\alpha \choose \beta}(\partial^\beta v\cdot\nabla)\partial^{\alpha-\beta}v\,\partial^\alpha v \,\de x\\
&+\int_{\T^3}\sum_{\beta\leq \alpha}{\alpha\choose\beta}\left[(\partial^\beta u\cdot\nabla)\partial^{\alpha-\beta}v\,\partial^\alpha v+(\partial^\beta v\cdot\nabla)\partial^{\alpha-\beta}u\, \partial^\alpha v\right]\de x\\
&=\int_{\T^3}\sum_{\beta\leq \alpha}{\alpha\choose\beta}\left[(\partial^\beta h\cdot\nabla)\partial^{\alpha-\beta}h\,\partial^\alpha v+(\partial^\beta b\cdot\nabla)\partial^{\alpha-\beta}h\,\partial^\alpha v+ (\partial^\beta h\cdot\nabla)\partial^{\alpha-\beta}b\,\partial^\alpha v\right]\de x.
\end{align*}
We use the divergence-free condition and, by integration by parts and Young's inequality, we can estimate the terms above as follows
\begin{equation}
\left|\int_{\T^3} (\partial^\beta v\cdot\nabla)\partial^{\alpha-\beta}v\,\partial^\alpha v \,\de x\right|\leq \frac{\e}{6}\int_{\T^3}|\nabla\partial^\alpha v|^2\de x+C\int_{\T^3}|\partial^\beta v|^2|\partial^{\alpha-\beta} v|^2\de x
\end{equation}
\begin{equation}
\left|\int_{\T^3} (\partial^\beta u\cdot\nabla)\partial^{\alpha-\beta}v\,\partial^\alpha v \,\de x\right|\leq \frac{\e}{6}\int_{\T^3}|\nabla\partial^\alpha v|^2\de x+C\int_{\T^3}|\partial^\beta u|^2|\partial^{\alpha-\beta} v|^2\de x
\end{equation}
\begin{equation}
\left|\int_{\T^3} (\partial^\beta v\cdot\nabla)\partial^{\alpha-\beta}u\,\partial^\alpha v \,\de x\right|\leq \frac{\e}{6}\int_{\T^3}|\nabla\partial^\alpha v|^2\de x+C\int_{\T^3}|\partial^\beta v|^2|\partial^{\alpha-\beta} u|^2\de x
\end{equation}
\begin{equation}
\left|\int_{\T^3} (\partial^\beta h\cdot\nabla)\partial^{\alpha-\beta}h\,\partial^\alpha v \,\de x\right|\leq \frac{\e}{6}\int_{\T^3}|\nabla\partial^\alpha v|^2\de x+C\int_{\T^3}|\partial^\beta h|^2|\partial^{\alpha-\beta} h|^2\de x
\end{equation}
\begin{equation}
\left|\int_{\T^3} (\partial^\beta b\cdot\nabla)\partial^{\alpha-\beta}h\,\partial^\alpha v \,\de x\right|\leq \frac{\e}{6}\int_{\T^3}|\nabla\partial^\alpha v|^2\de x+C\int_{\T^3}|\partial^\beta b|^2|\partial^{\alpha-\beta} h|^2\de x
\end{equation}
\begin{equation}
\left|\int_{\T^3} (\partial^\beta h\cdot\nabla)\partial^{\alpha-\beta}b\,\partial^\alpha v \,\de x\right|\leq \frac{\e}{6}\int_{\T^3}|\nabla\partial^\alpha v|^2\de x+C\int_{\T^3}|\partial^\beta h|^2|\partial^{\alpha-\beta} b|^2\de x
\end{equation}
By summing up over $\beta$ the above inequalities we get that
\begin{align*}
\frac{1}{2}\frac{\de}{\de t}\int_{\T^3}|\partial^\alpha v|^2\de x+\nu\int_{\T^3}|\nabla\partial^\alpha v|^2\de x & \leq \e\int_{\T^3}|\nabla\partial^\alpha v|^2\de x+C\|u(t)\|_{L^\infty}^2 \int_{\T^3}|\partial^\alpha v|^2\de x\\
&+C\sum_{s=0}^{r-1}\int_{\T^3}|\nabla^s v|^2|\nabla^{r-s} v|^2 \de x+C\sum_{s=1}^{r}\int_{\T^3}|\nabla^s u|^2|\nabla^{r-s} v|^2 \de x\\
&+C\|b(t)\|_{L^\infty}^2 \int_{\T^3}|\partial^\alpha v|^2\de x+C\sum_{s=0}^{r-1}\int_{\T^3}|\nabla^s h|^2|\nabla^{r-s} h|^2 \de x\\
&+C\sum_{s=1}^{r}\int_{\T^3}|\nabla^s b|^2|\nabla^{r-s} h|^2 \de x.
\end{align*}
We use Gagliardo-Nieremberg's inequality to estimate
\begin{equation}
\|\nabla^s v\|_{L^\infty}\leq C \|\nabla^{s+2} v\|^{1/2}_{L^2}\|\nabla^s v\|^{1/2}_{L^6}\leq C \|\nabla v\|^{1/2}_{H^{s+1}}\|v\|^{1/2}_{H^{s+1}},
\end{equation}
and the definition of $e_r$ to get
\begin{align*}
\frac{1}{2}\frac{\de}{\de t}\int_{\T^3}|\partial^\alpha v|^2\de x+\nu\int_{\T^3}|\nabla\partial^\alpha v|^2\de x & \leq \e\int_{\T^3}|\nabla\partial^\alpha v|^2\de x+C(\|u(t)\|_{L^\infty}^2+\|b(t)\|_{L^\infty}^2)e_r(t) \\
&+C (\|\nabla v(t)\|_{H^r}+\|\nabla h(t)\|_{H^r})e_r(t)^{3/2}\\
&+C\sum_{s=1}^{r}(\|\nabla^s u(t)\|_{L^\infty}^2+\|\nabla^s b(t)\|_{L^\infty}^2)e_{r-s}(t)\\
&\leq \e\int_{\T^3}|\nabla\partial^\alpha v|^2\de x+C(\|u(t)\|_{L^\infty}^2+\|b(t)\|_{L^\infty}^2)e_r(t) \\
&+\e(\|\nabla v(t)\|_{H^r}^2+\|\nabla h(t)\|_{H^r}^2)+Ce_r(t)^3\\
&+C\sum_{s=1}^{r}(\|\nabla^s u(t)\|_{L^\infty}^2+\|\nabla^s b(t)\|_{L^\infty}^2)e_{r-s}(t),
\end{align*}
where in the last computation we applied Young's inequality. Finally, by using Poincar\'e's inequality and summing over $\alpha$, we end up to
\begin{align*}
\frac{\de}{\de t}\sum_{|\alpha|\leq r}\int_{\T^3}|\partial^\alpha v|^2\de x &\leq -2(\nu-\e)\sum_{|\alpha|\leq r} \int_{\T^3}|\nabla\partial^\alpha v|^2\de x+Ce_r(t)^3\\
&+C(\|u(t)\|_{L^\infty}^2+\|b(t)\|_{L^\infty}^2)e_r(t)\\
&+C\sum_{s=0}^{r-1}\left(\|\nabla^{r-s}u(t)\|_{L^\infty}^2+\|\nabla^{r-s}b(t)\|_{L^\infty}^2\right)e_s(t).\\
\end{align*}
We now consider the equation for the magnetic field: we apply $\partial^\alpha$ to the equation obtaining
\begin{align*}
\partial_t \partial^\alpha h-\eta\Delta\partial^\alpha h & +\sum_{\beta \leq \alpha} {\alpha\choose\beta} \left((\partial^\beta v\cdot\nabla)\partial^{\alpha-\beta}h+(\partial^\beta u\cdot\nabla)\partial^{\alpha-\beta}h +(\partial^\beta v\cdot\nabla)\partial^{\alpha-\beta}b\right)\\
& =\sum_{\beta \leq \alpha} {\alpha\choose\beta} \left((\partial^\beta h\cdot\nabla)\partial^{\alpha-\beta}v+(\partial^\beta h\cdot\nabla)\partial^{\alpha-\beta}u+(\partial^\beta b\cdot\nabla)\partial^{\alpha-\beta} v \right).
\end{align*}
By multiplying by $\partial^\alpha h$ we get
\begin{align*}
\frac{1}{2}\frac{\de}{\de t} \int_{\T^3}|\partial^\alpha h|^2&\de x+\eta\int_{\T^3}|\nabla\partial^\alpha h|^2\de x \\ &+\sum_{\beta \leq \alpha} {\alpha\choose\beta} \int_{\T^3}\left((\partial^\beta v\cdot\nabla)\partial^{\alpha-\beta}h+(\partial^\beta u\cdot\nabla)\partial^{\alpha-\beta}h +(\partial^\beta v\cdot\nabla)\partial^{\alpha-\beta}b\right)\cdot\partial^\alpha h \de x\\
& =\sum_{\beta \leq \alpha} {\alpha\choose\beta} \int_{\T^3}\left((\partial^\beta h\cdot\nabla)\partial^{\alpha-\beta}v+(\partial^\beta h\cdot\nabla)\partial^{\alpha-\beta}u+(\partial^\beta b\cdot\nabla)\partial^{\alpha-\beta} v \right)\cdot\partial^\alpha h\de x.
\end{align*}
We now use estimates similar to those above, where $\nabla^\alpha h$ plays the role of $\nabla^\alpha v$, obtaining
\begin{align*}
\frac{\de}{\de t}\sum_{|\alpha|\leq r}\int_{\T^3}|\partial^\alpha h|^2\de x&\leq -2(\eta-\e)\sum_{|\alpha|\leq r} \int_{\T^3}|\nabla\partial^\alpha h|^2\de x\\
&+C\left(\|u(t)\|_{L^\infty}^2+\|b(t)\|^2_{L^\infty}\right)e_r(t)+C e_r(t)^3\\
&+\sum_{s=0}^{r-1}\left(\|\nabla^{r-s}u(t)\|^2_{L^\infty}+\|\nabla^{r-s}b(t)\|^2_{L^\infty}\right)e_s(t).
\end{align*}
By summing the inequalities we get
\begin{align}
\frac{\de}{\de t}e_r(t)\leq &-2(\nu+\eta-2\e)e_r(t)+C\left(\|u(t)\|_{L^\infty}^2+\|b(t)\|^2_{L^\infty}\right)e_r(t)+C e_r(t)^3 \nonumber\\
&+2\sum_{s=0}^{r-1}\left(\|\nabla^{r-s}u(t)\|^2_{L^\infty}+\|\nabla^{r-s}b(t)\|^2_{L^\infty}\right)e_s(t)\label{stab-final}
\end{align}
Let us assume that $e_r(t)^2$ is small enough that 
$$
-2(\nu+\eta-2\e)e_r+C e_r^3\leq -2\sigma e_r.
$$
Then, by substituting in \eqref{stab-final} and using the inductive step on $e_s$ we obtain (recall that here $s \geq 1$)
\begin{align*}
& \frac{\de}{\de t}e_r(t)\leq  -2\sigma e_r+C(\|u(t)\|_{L^\infty}^2+\|b(t)\|_{L^\infty}^2)e_r\\
&+C N^{2s-2} \exp\left(-2\sigma t+C\int_0^t\|u(\tau)\|^2_{L^\infty}+\|b(\tau)\|^2_{L^\infty}\de \tau\right)
\sum_{s=0}^{r-1}\left(\|\nabla^{r-s}u(t)\|_{L^\infty}^2+\|\nabla^{r-s}b(t)\|_{L^\infty}^2\right)Q_s(t)e_s(0),
\end{align*}
and hence by Gronwall's lemma
\begin{align}\nonumber
e_r(t)& \leq C \exp\left(-2\sigma t+C(\|u\|^2_{L^2L^\infty}+\|b\|^2_{L^2L^\infty})\right)e_r(0)\left[ 1+\sum_{s=0}^{r-1}
N^{2s-2} \left(\|\nabla^{r-s}u\|_{L^2L^\infty}^2+\|\nabla^{r-s}b\|_{L^2L^\infty}^2\right)\right]
\\ \label{aPrioriBoundFinal}
&\leq C (1 + N^{2r-2}) \exp\left(-2\sigma t+C(\|u\|^2_{L^2L^\infty}+\|b\|^2_{L^2L^\infty})\right)e_r(0),
\end{align}
where in the second inequality we used assumption \eqref{AssOnWInfty}. 

Now, the assumption we made on $e_r(t)$ is satisfied at time $t=0$ (and thus for short times) by the smallness hypothesis \eqref{assOnPertData3d}. Moreover, this can be extended to any further time using the inequality \eqref{aPrioriBoundFinal} and again  \eqref{assOnPertData3d}. This completes the proof.
%
\end{proof}

\section{Magnetic reconnection in 3D}\label{Sec:Proof3D}
In this section we provide an example of magnetic reconnection in the three-dimensional case. The idea is to construct a global solution of \eqref{eq:mhd} via a perturbative argument in such a way that we can ``control" the topology of its magnetic lines at $t=0$ and $t=T$.
\begin{proof}[Proof of Theorem \ref{thm:main}, d=3] We divide the proof in several steps.\\
\\
\underline{\em Step 1} \hspace{0.3cm}Construction of the global smooth solution.\\
\\
Let $B_0,B_1$ be the Beltrami fields defined in \eqref{b0} and in Theorem \ref{Thmdef:BN1} respectively, and consider the couple $(u_0,b_0)=(0,MB_0)$ as an initial datum for \eqref{eq:mhd}. Then, a global smooth solution of \eqref{eq:mhd} is given by 
\begin{equation}
(u(t,x),b(t,x))=\left(0,Me^{-\eta N_0^2t}B_0(x)\right),
\end{equation}
with pressure
$$
P_{u,b}(t,x) = - \frac12 M^2e^{-2\eta N_0^2t}|B_0(x)|^2.
$$
This may be easily verified recalling that $B_0$ is a Beltrami field.
Now, consider the couple $(0, m_0)$ as initial datum of \eqref{eq:mhd}, where $m_0$ is given by
$$
m_0:=MB_0+\delta B_1.
$$
Looking at the definition of $\eqref{b0}$ we immediately see that
$$
\| MB_0 \|_{H^r} = M N_0^{r} 
$$ 
We want to construct a global smooth solution $(w,m)$ 
starting from such an initial datum. First of all, we define
$$
h_0:=m_0 - b_0=\delta B_1.
$$
We compute the $H^r$ norm of $\delta B_1$ using \eqref{B-norm}, the fact that $B_1$ is Beltrami 
and the equivalence between the $H^1$ norm of $B_1$ and the $L^2$ norm of its curl, so that we arrive to 
\begin{equation}
C N_1^{r-1}< \| B_1\|_{L^2}< C N_1^{r - \frac12}.
\end{equation}
Thus in particular we have that
\begin{equation}\label{cond:1Preq}
\|h_0\|_{H^s} \leq C\delta N_1^s.
\end{equation}
Since we want to apply Theorem \ref{thm:stab}, we require that
\begin{equation}\label{cond:1}
\delta N_1^r \ll N_0^{1-r}.
\end{equation}
Thus, recalling also Remark \ref{RemarkAfterCor}, 
we know that there exists a unique global solution $(w,m)$ starting from $(0,m_0)$ and the difference  
$h(t, \cdot) = m(t,\cdot) - b(t,\cdot) $ 
satisfies
\begin{equation}\label{est:stab_globInProof!}
\|w(t,\cdot)\|_{H^s}+\|h(t,\cdot)\|_{H^s} \leq C(N_0^{s-1}+1)e^{-\sigma t} \|h_0\|_{H^s}
\end{equation}
for all $0\leq s \leq r$.
\\
\\
\underline{\em Step 2} \hspace{0.3cm} Further estimates on the global solution.\\
\\
We need more estimates on the difference $h$ in order to control the behavior of the fluid at time $t=T$.
Recall that, since our reference solution if $u=0$, here we have $w=v$.
\begin{equation}
\label{eq:duhamel-h}
h(t,\cdot)=e^{\eta t\Delta} h_0+\int_0^t e^{\eta (t-s)\Delta}\dive \big(h(s)\otimes v(s)-v(s)\otimes h(s)+b(s)\otimes v(s)-v(s)\otimes b(s)\big)\,\de s
\end{equation}
For simplicity, we define
\begin{equation}\label{Def:Lh}
L_h(t,\cdot):=\int_0^t e^{\eta(t-s)\Delta}\dive \big(h(s)\otimes v(s)-v(s)\otimes h(s)\big)\,\de s,
\end{equation}
\begin{equation}\label{Def:Lb}
L_b(t,\cdot):=\int_0^t e^{\eta(t-s)\Delta}\dive \big(b(s)\otimes v(s)-v(s)\otimes b(s)\big)\,\de s.
\end{equation}
First of all, note that by Theorem \ref{thm:stab} and the estimate \eqref{cond:1} of the previous step, we get that
\begin{align}
\|v(t,\cdot)\|_{H^r}+\|h(t,\cdot)\|_{H^r}&\leq C N_0^{r-1} e^{-\sigma t} \|h_0\|_{H^r} \nonumber\\
&\leq \delta N_0^{r-1}N_1^r;
\end{align}
recall that here we have chosen $v_0=0$. Then, by using the above formula, we estimate the tensorial products in $L_h,L_b$ as follows
\begin{align}
\| h(s)\otimes v(s)\|_{H^{r+1}}&\leq \| h(s)\|_{L^\infty} \| v(s)\|_{H^{r+1}}+ \| v(s)\|_{L^\infty}\| h(s)\|_{H^{r+1}} \nonumber\\
&\leq \| h(s)\|_{H^2} \| v(s)\|_{H^{r+1}}+ \| v(s)\|_{H^2}\| h(s)\|_{H^{r+1}}\nonumber\\
&\leq C \delta^2 N_0^{r+1}N_1^{r+3} e^{-\sigma s},\\
\| b(s)\otimes v(s)\|_{L^2}&\leq \|b(s)\|_{L^\infty}\|v(s)\|_{L^2}\leq C \delta e^{-\eta N_0^2 s},\\
\| b(s)\otimes v(s)\|_{H^{r+1}}&\leq C\| v(s)\|_{L^\infty} \|b(s)\|_{H^{r+1}}+ C\| v(s)\|_{H^{r+1}}\|b(s)\|_{L^\infty}\nonumber\\
&\leq C\| v(s)\|_{H^2} \|b(s)\|_{H^{r+1}}+ C\| v(s)\|_{H^{r+1}}\|b(s)\|_{L^\infty}\nonumber\\
&\leq C e^{-\eta N_0^2 s}\big(\delta N_0^{r+2}N_1^{2}  +\delta N_0^{r} N_1^{r+1}\big) \nonumber \\
&\leq Ce^{-\eta N_0^2 s}\delta N_0^{r+2} N_1^2,
\end{align}
where in the last inequality we assumed that
\begin{equation}\label{cond:2}
N_0^2 \gg N_1^{r-1}.
\end{equation}
Using (see \cite{ELP}, Lemma 4.3)
\begin{equation}\label{UT1Fin}
\|e^{s\Delta}f \|_{H^m} \leq e^{-s}\|f\|_{H^m}, \qquad \mbox{for $f$ with zero-average}, 
\end{equation}
we estimate $L_h$ as follows
\begin{align}
\|L_h(t,\cdot)\|_{H^r}&\leq C\int_0^t \|e^{\eta(t-s)\Delta} \big(h(s)\otimes v(s)\big)\|_{H^{r+1}}\de s\nonumber\\
&\leq C\int_0^t e^{-\eta(t-s)} \|h(s)\otimes v(s)\|_{H^{r+1}}\de s\nonumber\\
& \leq  C\delta^2 N_0^{r+1}N_1^{r+3}\int_0^t e^{-\eta(t-s)} e^{-\eta\sigma s}\de s\nonumber\\
&\leq  C\delta^2 N_0^{r+1}N_1^{r+3}.
\end{align}
On the other hand, using  (see \cite{ELP}, Lemma 4.3)
\begin{equation}\label{UTFBis}
\|e^{s\Delta}f \|_{H^m} \leq C s^{-m/2} \|f\|_{L^2}, \qquad \mbox{for $f$ with zero-average},
\end{equation}
we estimate
\begin{align}
\|L_b(t,\cdot)\|_{H^r}&\leq C\int_0^{t/2} \|e^{\eta(t-s)\Delta}\big(b(s)\otimes v(s)\big)\|_{H^{r+1}}\de s + C\int_{t/2}^t \|e^{\eta(t-s)\Delta}\big(b(s)\otimes v(s)\big)\|_{H^{r+1}}\de s\nonumber\\
&\leq C\int_0^{t/2} (t-s)^{-\frac{r+1}{2}} \|b(s)\otimes v(s)\|_{L^2}\de s + C\int_{t/2}^t e^{-\eta(t-s)} \|b(s)\otimes v(s)\|_{H^{r+1}}\de s\nonumber\\
&\leq C \delta\int_0^{t/2} (t-s)^{-\frac{r+1}{2}} e^{-\eta N_0^2 s}\de s + C\delta N_0^{r+2} N_1^2\int_{t/2}^t e^{-\eta(t-s)}e^{-\eta N_0^2 s}\de s\nonumber\\
&\leq C\delta N_0^{-2}+ C\delta N_0^{r+2} N_1^2.
\end{align}
\\
\\
\underline{\em Step 3} \hspace{0.3cm} Choice of the parameters.\\
\\
In this step we fix the parameters $N_0,N_1,\delta$. First of all, we know that at time $t=0$ the magnetic field $m_0$ satisfies
$$
m_0=MB_0+\delta B_1.
$$
Then, let us consider the rescaled vector field
$$
M^{-1}m_0,
$$
it satisfies 
$$
\|M^{-1}m_0 -B_0\|_{H^r}\leq C \delta N_1^r\ll N_0^{1-r}.
$$
Thus, recalling \eqref{quantstab2} and property $(i)$ of Section \ref{Sec:BeltramiTaylor}, using the Sobolev embedding and taking $r$ sufficiently large, we see that the integral lines of $M^{-1}m_0$ are all non contractible. Furthermore, since the rescaling does not change the topology of its integral lines, all the integral lines of $m_0$ are contractible.\\
We now consider the behaviour of the fluid at time $t=T$. We rescale the magnetic field as
\begin{equation}
\delta^{-1}e^{\eta N_1^2T} m(T,\cdot),
\end{equation}
and then since $m=b+h$ and by formula \eqref{eq:duhamel-h} we get 
\begin{align*}
\delta^{-1}&e^{\eta N_1^2T} m(T,\cdot)=B_1+\frac{M}{\delta} e^{-\eta(N_0^2-N_1^2)T}B_0(x)\\
&+\delta^{-1}e^{\eta N_1^2T}\int_0^T e^{\eta(T-s)\Delta}\dive \big(h(s)\otimes v(s)+b(s)\otimes v(s)-v(s)\otimes h(s)-v(s)\otimes b(s)\big)\,\de s
\end{align*}
Our goal is to choose the constant $\delta, N_0, N_1$ such that 
$$
\left\| \delta^{-1}e^{\eta N_1^2T} m(T,\cdot)-B_1 \right\|_{H^r} \ll 1, 
$$
 for a sufficiently large $r$. 
Thus, since $\delta^{-1}e^{\eta N_1^2T} m(T,\cdot)$ is simply a rescaling of $m(T,\cdot)$, the integral lines of $m(T,\cdot)$ will be locally diffeomorphic to the set $\mathcal{S}$ defined in Section \ref{Sec:BeltramiTaylor}, as consequence of \eqref{quantstab}, property $(ii)$ and Sobolev embedding. In particular $m(T,\cdot)$ will possess some contractible integral lines, thus the set of its integral lines will be not homotopically equivalent to that of $B_1$, proving that magnetic reconnection must have happened between the times $t=0$ and $t=T$.\\
It remains to show that $\tilde{b}_1$ verify \eqref{quantstab}.  By the estimates proved in the Step 2 we have that
\begin{align}
\|M\delta^{-1} e^{-\eta(N_0^2-N_1^2)T}B_0\|_{H^r}&\leq M\delta^{-1} e^{-\eta(N_0^2-N_1^2)T} \|B_0\|_{H^r}\nonumber\\
&\leq M\delta^{-1} e^{-\eta(N_0^2-N_1^2)T} N_0^r,\label{cond:3}
\end{align}

\begin{align}
\|\delta^{-1}e^{\eta N_1^2T} L_h(T,\cdot)\|_{H^r}&\leq C \delta^{-1}e^{\eta N_1^2T} \|L_h(T,\cdot)\|_{H^r}\nonumber\\
&\leq C \delta^{-1}e^{\eta N_1^2T} \delta^2 N_0^{r+1}N_1^{r+3}\nonumber\\
&= C\delta N_0^{r+1}N_1^{r+3}e^{\eta N_1^2T},\label{cond:4}
\end{align}

\begin{align}
\|\delta^{-1}e^{\eta N_1^2T} L_b(T,\cdot)\|_{H^r}&\leq \delta^{-1}e^{\eta N_1^2T}\|L_b(T,\cdot)\|_{H^r}\nonumber\\
&\leq \delta^{-1}e^{\eta N_1^2T}\delta N_0^{-2}\nonumber\\
&=e^{\eta N_1^2T}N_0^{-2}.\label{cond:5}
\end{align}
In order to make the above quantities small, we define $\delta$ as
\begin{equation}
\delta:= N_0^{-(r+1)}N_1^{-(r+3)}e^{-2\eta N_1^2T}.
\end{equation}
With this choice, we can easily verify that \eqref{cond:1} holds. Moreover, by the choice of $\delta$ as above, we have that \eqref{cond:3} is small if
\begin{equation}\label{cond:3'}
e^{-\eta N_0^2 T}\ll N_0^{-2r-1}N_1^{-r-3} M^{-1} e^{-3\eta N_1^2T}.
\end{equation}
Note that if $N_0$ is chosen sufficiently larger than $N_1$, \eqref{cond:2} and \eqref{cond:3'} are satisfied and the quantities in \eqref{cond:4} and \eqref{cond:5} can be made arbitrarily small.\\
\\
\underline{\em Step 4}\hspace{0.5cm} Rescaling of the initial datum.\\
\\
Lastly, to complete the proof of the theorem, we need to rescale the norm of $m_0$ in the above construction. This can be done by replacing it with the initial condition $\frac{Mm_0}{\|m_0\|_L^2}$, since the rescaling factor does not change anything in the above arguments.
\end{proof}

Some remarks on the proof above are in order. 

\begin{rem}
\label{Remark:LargeDataWithStructure1}
The choice $u_0=0$ simplifies the proof but we may easily generalize the argument to small velocities, namely taking 
$\| u_0 \|_{H^r} = \varepsilon$, where the size of the small parameter $\e$ depends on all the relevant parameters we introduced in the proof. Moreover, we may consider large data $u_0$ (and $\varepsilon$ perturbation of them) if we add an appropriate structure. For example if we choose $u_0 = M B_0$ it is easy to see that
$$
\left(Me^{-\nu N_0^2t}B_0(x),Me^{-\eta N_0^2t}B_0(x)\right),
$$
is a smooth solution of \eqref{eq:mhd} with initial datum $(u_0,b_0)=(MB_0,MB_0)$, choosing the pressure
$$
P= \frac{M^2}{2} ( e^{- 2 \nu N_0^2t} -  e^{-2\eta N_0^2t})|B_0|^2.
$$
Then, when we estimate the $H^r$-norm of \eqref{eq:duhamel-h} there are additional terms to be computed, i.e.
$$
\int_0^t e^{\eta (t-s)\Delta}\dive \big(h(s)\otimes u(s)-u(s)\otimes h(s)\big) \de s,
$$
which can be analyzed in the same way we did for $L_b$ to close the argument. 
\end{rem}
\begin{rem}\label{Remark:LargeDataWithStructure2}
In the computations above it is crucial that $\eta>0$, indeed the choice of the frequency $N_0 = N_0(\eta)$ depends on $\eta$, and $N_0(\eta) \to \infty$ as $\eta\to 0$, preventing to promote the magnetic reconnection scenario to the ideal case (in fact we choose $N_{0}$ proportional to $\eta^{-1/2}$). 
\end{rem}
\begin{rem}\label{Remark:LargeDataWithStructure3}
Differently to $\eta$, the estimates above do not blow up in the vanishing viscosity limit $\nu\to 0$ and one could in principle prove a similar (local in time) reconnection statement for $\nu =0$. 
\end{rem}

\section{Taylor fields}\label{Sec:BeltramiTaylorBis}

After fixing the notations and recalling some definitions and structurally stability results, we introduce 
the main mathematical objects which are needed in our 2D magnetic reconnection proof. These are the 
so-called Taylor fields, which 
may be viewed as a 2D counterpart of the Beltrami fields from Section \ref{Sec:BeltramiTaylor}.

 We say that $v:\T^2 \to \R^2$ is a {\em Hamiltonian} vector field if it can be express as the orthogonal gradient of a scalar function $\psi$, i.e.
$$
v=\nabla^\perp \psi:=(\partial_{x_2}\psi,-\partial_{x_1}\psi).
$$
Hamiltonian vector fields are by definition divergence-free and we denote by
$$
D^r_H(\T^2)=\{ v\in C^r(\T^2): v \mbox{ is a Hamiltonian vector field}\}.
$$
We recall that a singular point $x_0$ of a vector field $v\in C^r(\T^2)$ is said to be {\em non-degenerate} if $\nabla v(x_0)$ is an invertible matrix. It is worth to note that if $v$ is a smooth divergence-free vector field, then
a non-degenerate singular point of $v$ must be either a saddle or a center.

\begin{rem}
By the Helmoltz decomposition, an incompressible vector field on $\T^2$ is either Hamiltonian or a constant vector field. 
Since we will always work with zero-average vector fields for us there will be no difference between being 
incompressible and Hamiltonian. 
\end{rem}

We now give the definition of structural stability.
\begin{defn}
A vector field $v$ on $\T^d$ is structurally stable if there is a neighborhood $\mathcal{U}$ of $v$ in $C^1(\T^d)$ such that whenever $v'\in \mathcal{U}$ there is a homeomorphism of $\T^d$ onto itself transforming trajectories of $v$ onto trajectories of $v'$.
\end{defn}
Note that, as a consequence of the classical result of Peixoto \cite{Pe}, no 2D divergence-free vector field with 
some critical point that is a center is structurally stable under general $C^r$ perturbations. 
This is because centers might be destroyed adding small sink or sources, which in the incompressible setting are however 
forbidden. Thus, in the divergence-free case a stronger result holds, see \cite{MW}.
Moreover, the Peixoto stability result does not allow any kind of connection between saddle points, while in 
the incompressible 
case saddle self-connections are allowed. On the other hand, heteroclinic self connections are 
expected to give rise to bifurcation phenomena also in the incompressible case (see Theorem \ref{thm:main3}).

\begin{thm}[Ma, Wang \cite{MW}]\label{thm:MW}
A divergence-free Hamiltonian vector field $v\in D^r_H(\T^2)$ with $r\geq 1$ is structurally stable under Hamiltonian vector field perturbations if and only if
\begin{itemize}
\item $v$ is regular, i.e. all singular points of $v$ are not degenerate;
\item all saddle points of $v$ are self-connected.
\end{itemize}
Moreover, the set of all $C^r$ structurally stable Hamiltonian vector fields is open and dense in~$D^r_H(\T^2)$.
\end{thm}

\subsection{Building blocks in 2D: Taylor vortices}\label{Subsec:TaylorV} For the two-dimensional case we need to introduce a class of vector fields called {\em Taylor vortices}. We refer to \cite{MWbook} for more details on the following. We consider the eigenvectors of the following Stokes problem
\begin{equation}\label{eq:eigen-stokes}
\begin{cases}
-\Delta \mathcal V=\lambda  \mathcal V,\\
 \mathcal V=\nabla^{\perp} \psi.
\end{cases}
\end{equation}
Provided that $\lambda\in \N$, for any couple of integers $n\geq 1,m\geq 0$ we can easily construct a solution of \eqref{eq:eigen-stokes} with $\lambda=n^2+m^2$ in the following way 
\begin{align*}
 \mathcal V_{nm}^1 &= (m\sin nx_1 \cos mx_2, -n\cos n x_1 \sin m x_2), &
 \mathcal V_{nm}^2 & = (m\cos nx_1\sin mx_2, -n\sin nx_1 \cos m x_2),\\
 \mathcal V_{nm}^3 & = (m\cos nx_1\cos mx_2, n\sin n x_1 \sin m x_2), &
 \mathcal V_{nm}^4 & = (m\sin n x_1 \sin m x_2, n\cos n x_1\cos m x_2),\\
 \mathcal V_{n}^1 & = (\sin n x_2,0),\hspace{0.5cm}  \mathcal V_{n}^2  = (\cos n x_2, 0), &
 \mathcal V_{n}^3 & = (0,\sin n x_1),\hspace{0.5cm}  \mathcal V_{n}^4  = (0, \cos n x_1).
\end{align*}
Denote by $\mathbb{V}_{nm} = \mathrm{span} \{ \mathcal V_{nm}^1, \mathcal V_{nm}^2,  \mathcal V_{nm}^3,  \mathcal V_{nm}^4 \}$ and $\mathbb{V}_{n} = \mathrm{span} \{ \mathcal V_{n}^1,  \mathcal V_{n}^2,  \mathcal V_{n}^3,  \mathcal V_{n}^4 \}$. Then, by varying $n,m$, these spaces generate all the
(zero-average) solutions of \eqref{eq:eigen-stokes}.

It is important to note that the vector fields $ \mathcal V_{nm}$ (and $ \mathcal V_n$ respectively) are stationary solution of the Euler equations
for a suitable choice of the pressure. For instance we have
\begin{equation}
\begin{cases}
( \mathcal V^1_{nm} \cdot \nabla)  \mathcal V^1_{nm} = \nabla P^1_{nm},\\
\dive  \mathcal V^1_{nm}=0,
\end{cases}
\end{equation}
with pressure given by
$$
P^1_{nm} = m^2 (\sin (nx_1))^2 + n^2 (\sin (m x_2))^2.
$$

Crucially, all vector fields in $\mathbb{V}_{nm}$ are not Hamiltonian structurally stable, because they are not saddle self-connected, see Figure \ref{fig:v22}. However, they may have two kind of topological structure, see \cite[Thorem 4.5.3]{MWbook}. 

\begin{figure}[H]
\centering
\includegraphics[width=0.5\textwidth]{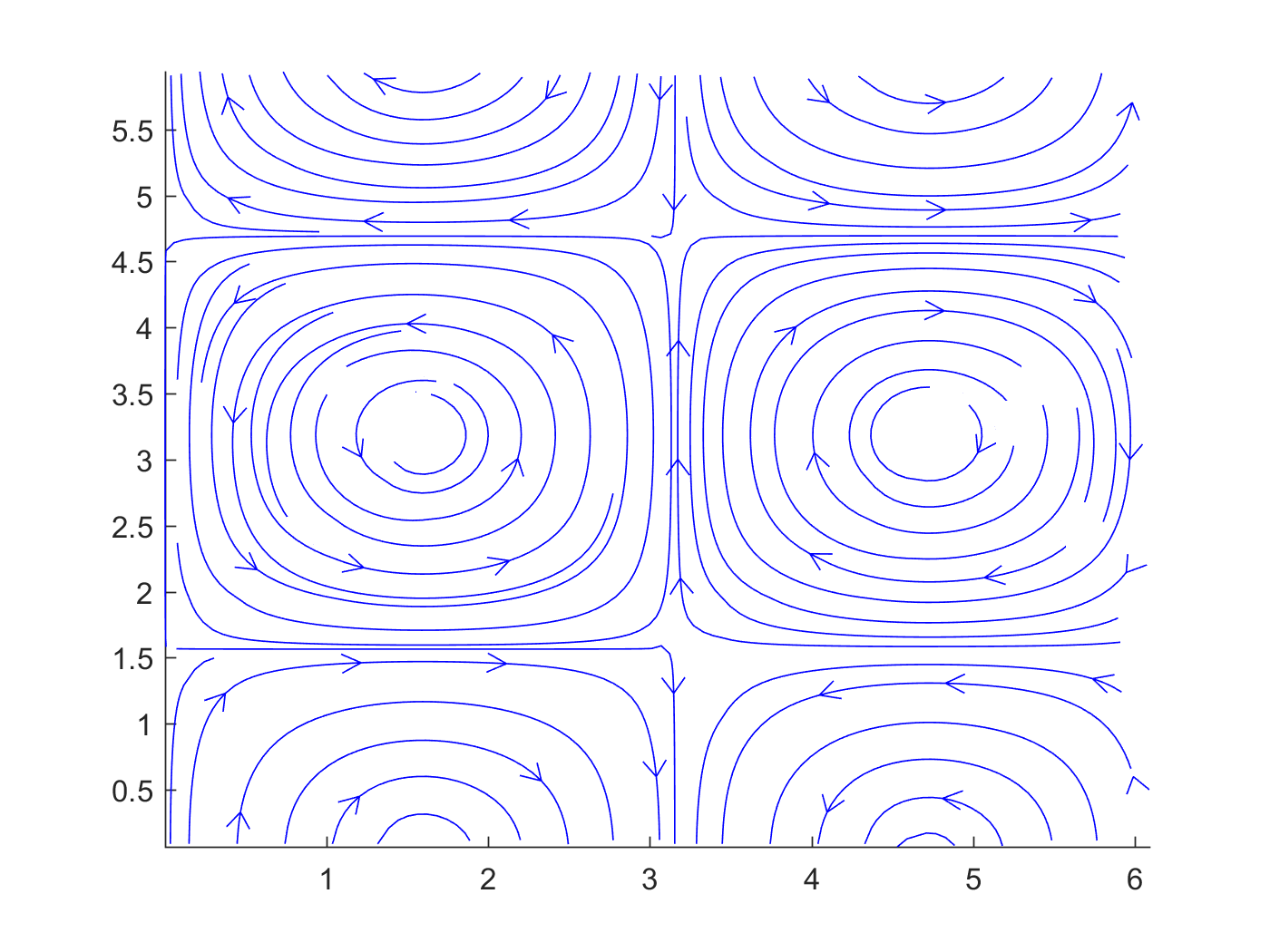}
\caption{Phase diagram of the vector field $\mathcal{V}^4_{11}$.}
\label{fig:v22}
\end{figure}

Let us recall that the set of Hamiltonian structurally stable vector field is open and dense in the set of Hamiltonian vector field. This means that, for any given $ \mathcal V_{nm}$ one can find a neighborhood $\mathcal{U}$, w.r.t. the topology in $C^1$, such that if $ \mathcal V \in \mathcal{U}$ then $ \mathcal V$ is a Hamiltonian structurally stable vector field, see \cite[Theorem 3.3.2]{MWbook}. This in particular means that by perturbing $ \mathcal V_{nm}$ it is possible to break 
heteroclinic saddle connections. 
This is suggested by the fact that for Hamiltonian perturbations the Melnikov function associated to an 
heteroclinic saddle connection 
may be non zero, while it has to be zero in the case of homoclinic saddle connection (since the perturbation is autonomous,
the Melnikov function is indeed a constant in this cases).
A rigorous mechanism to break heteroclinic saddle connections has been proposed in \cite{MW}.
The same holds for the space $\mathbb{V}_n$ with $n>1$, while in $\mathbb{V}_1$ there are vector fields which are structurally stable. An example of a stable Taylor field is given by
\begin{equation}\label{def:v1}
V_1=\left(\sin x_2,\frac{1}{2}\sin x_1\right),
\end{equation}
whose phase diagram can be found in Figure \ref{fig:v1}. In particular, note that the structural stability follows from 
the absence of heteroclinic connections. The field $V_1$ has indeed two saddle points which 
are not connected by any integral line. It is also important for us that 
$V_1$ solves the stationary Euler $(V_1 \cdot \nabla )V_1 = \nabla P$ equation with pressure 
$P = \frac12 \cos x_1 \cos x_2$.

\begin{figure}[H]
\centering
\includegraphics[width=0.5\textwidth]{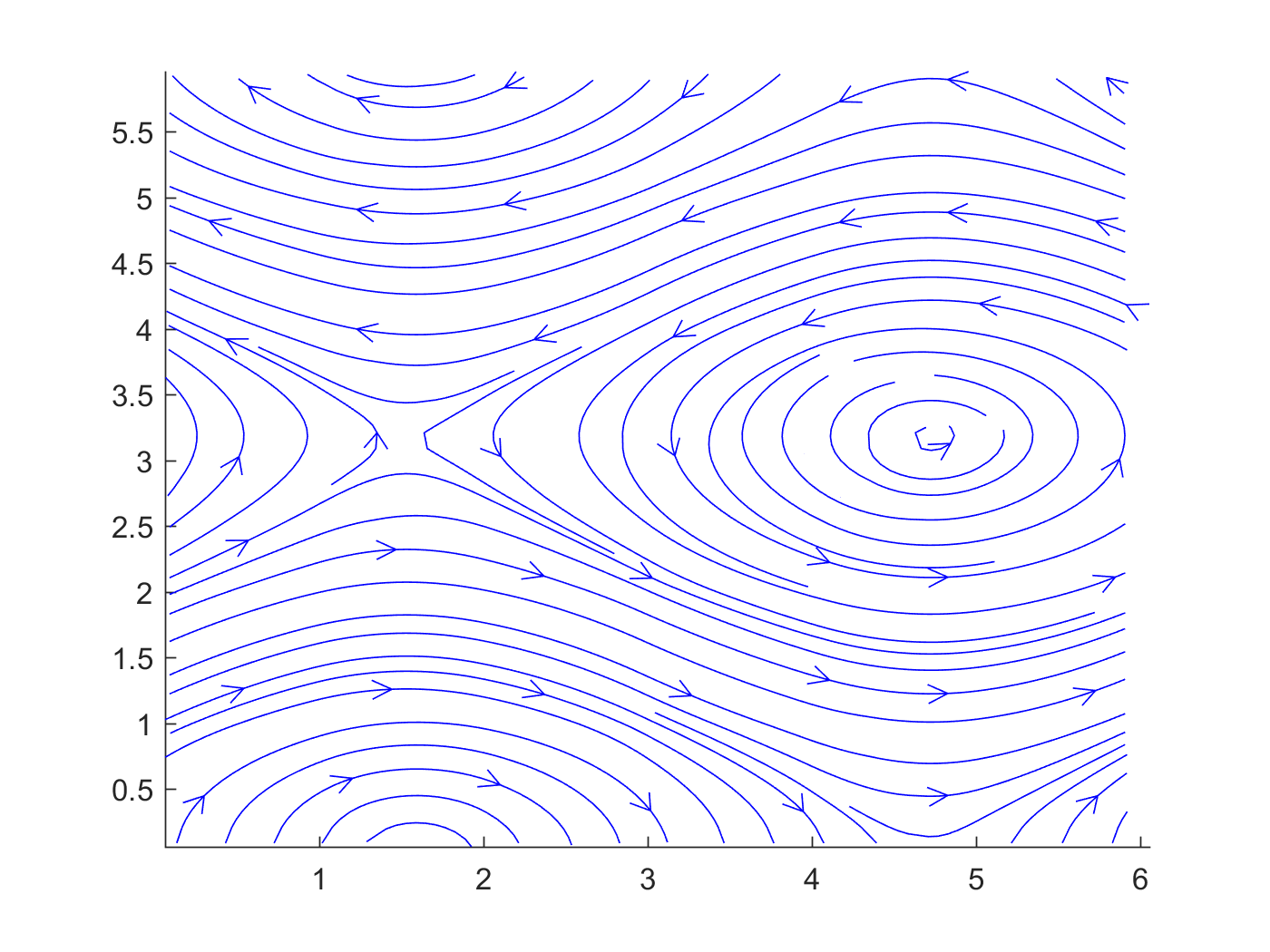}
\caption{Phase diagram of the vector field $V_{1}$.}
\label{fig:v1}
\end{figure}

Consider $n,m\geq 1$. We set $N^2 := n^2+m^2$ (note that here $N$ may be not integer). 
We denote by $ V_N =  \mathcal V_{nm}^1$. In the rest of the paper we will focus, to fix ideas, on this particular 
subfamily of Taylor fields with eigenvalue $N^2$, however we could similarly consider any other 
Taylor fields with eigenvalue $N^2$ with straightforward modifications of the arguments.
The critical points of $V_N$ are given by
\begin{equation}\label{CritPointsAre}
x^* := \left(\frac{k_1\pi}{n},\frac{k_2\pi}{m}\right)\hspace{0.5cm}\mbox{and}
\hspace{0.2cm} \bar{x} := \left(\frac{\pi}{2n}+\frac{k_1\pi}{n},\frac{\pi}{2m}+\frac{k_2\pi}{m}\right),
\end{equation}
with $k_1=1,...,2n$ and $k_2=1,...,2m$ (with a small abuse of notation we omit the dependence of $x^*$ and 
$\bar{x}$ on $k_1, k_2$, that will be however irrelevant). By computing the gradient of $V_N$ we obtain that
\begin{equation}
\nabla V_N=\begin{pmatrix}
nm\cos(nx_1)\cos(mx_2) & -m^2\sin(nx_1)\sin(mx_2)\\
n^2\sin(nx_1)\sin(mx_2) & -nm \cos(nx_1)\cos(mx_2)
\end{pmatrix}
\end{equation}
Thus at the critical points we have
\begin{equation}\label{Hyp/Ell}
\nabla V_N(x^*) = \pm \begin{pmatrix}
nm & 0\\
0 & -nm 
 \end{pmatrix}
\hspace{0.5cm}\mbox{or }\hspace{0.2cm}
\nabla V_N(\bar{x}) = \pm \begin{pmatrix}
0 & -m^2\\
n^2 & 0,
\end{pmatrix}
\end{equation}
and
$\nabla V_N$ is always invertible in the critical points with determinant $\pm \,n^2m^2$. 
Note that the $x^*$ are hyperbolic critical points, while the $\bar{x}$ are elliptic ones.

It is also worth to note that the distance between critical points can be controlled by below 
with~$\frac{C}{N}$ (this is clear by \eqref{CritPointsAre}). Moreover, a straightforward computation on the norms 
\begin{equation}\label{scalingvn}
\|v_N\|_{C^k}=c_k N^{k+1}. 
\end{equation}

We conclude this section with the following lemma, which investigate the stability of the critical points of 
our Taylor fields under small perturbations. If we do not take into account the quantitative bound \eqref{LowerBoundCP} 
for $\delta_0(N)$ the
content of the lemma is an immediate consequence of 
the implicit function theorem, as the linearization of $V_N$ is invertible at the critical points. However, 
it will be crucial for us to quantify the size of the perturbative parameter $\delta$, since in our applications 
it will not be possible to have $\delta$ too small (compared to a certain function of the frequency $N$). 
For instance, if in the next lemma we would allow the choice $\delta_0(N) = e^{-C N^2}$ with $C \gg 1$, 
we could not use it anymore in the proof of
Theorem \ref{thm:main}.  
 This is why in the 
proof below we run again the implicit function theorem machinery for this particular example.

\begin{lem}\label{Lemma:CriticalPoints}
Let $V_N$ be the Taylor field with eigenvalue $N^2=n^2+m^2$ defined above and let $ W \in C^1$. 
%
For all $N$ sufficiently large (depending only on $\| W \|_{C^1}$), there exists $\delta_0=\delta_0(N)$ such that the vector field $\tilde{V}:\T^2\to\R^2$ defined as
\begin{equation}
\tilde{V}(x):=V_N(x)+\delta  W(x),
\end{equation}
has at least $8nm$ regular critical points for every $|\delta|<\delta_0(N)$. We may choose 
\begin{equation}\label{LowerBoundCP}
\delta_0(N) =  N^{-L},
\end{equation} where~$L$ is a fixed large integer. 
\end{lem}
\begin{proof}
First of all, note that the critical points of $\tilde{V}$ are zeros of the vector-valued function $F:\T^2\times\R\to\R^2$ defined as
\begin{equation}\label{DefFPerturbation}
F(x,\delta)=V_N(x)+\delta W(x).
\end{equation}
At the hyperbolic critical points $x^*$ we have (see \ref{Hyp/Ell}) 
$$
F(x^*,0)=0,\hspace{0.2cm}\mbox{and} \hspace{0.2cm} \nabla_x F(x^*,0)\hspace{0.2cm} \mbox{is invertible}
$$
and the same is true at the elliptic critical points. 
We can thus apply the implicit function theorem to the function $F$ in a neighborhood of $x^*$ (or $\bar{x}$): 
this implies that here exists $\delta_0=\delta_0(N)>0$ and a (smooth) curve $x:[-\delta_0,\delta_0]\to\T^2$ with $x(0)=x^*$ 
(or $x(0)=\bar x$) such that $F(x(\delta),\delta)=0$ for all $|\delta|<\delta_0$. We must now quantify the value of 
$\delta_0$ to be as in \eqref{LowerBoundCP}. This will imply, in particular, that the
the new critical points $x(\delta)$, obtained with different choices 
of $x^*$ and $\bar{x}$, do not collapse into each others.

We first consider the hyperbolic points. Note that
\begin{equation}\label{Hyp/EllDelta}
\nabla_x F(x^*, \delta) = \pm
\begin{pmatrix}
 nm & 0\\
0 & - nm 
 \end{pmatrix}
+ \delta \nabla W(x^*), 
 \end{equation}
thus 
\begin{equation}\label{FirstMathcalO}
\det \nabla_x F(x^*, \delta)  = 
 - n^2 m ^2 + \mathcal{O}(\delta nm),
\end{equation}
and so 
$$
 | \det \nabla_x F(x^*, \delta) | > \frac{N^2}{2},
$$
for $|\delta| < \delta_0$ with $\delta_0$ sufficiently small. Here we used that $ N \lesssim nm \lesssim  N^2$.
Hereafter all the $\mathcal{O}$ may also depend on $\| W \|_{C^1}$. For instance, 
in \eqref{FirstMathcalO} we should have written $\mathcal{O}(\delta nm \|W\|_{C^1})$. 
However we will always omit the dependence to simplify the notations.

This allows us to define a suitable of map $\Phi$ such that the  
the curve $x(\delta)$ will be constructed as the fixed point of this map. 
Let denote with $\Gamma_{\rho_0, \delta_0}(x^*, 0)$ the set of all the curves 
$$
 \delta\in [-\delta_0, \delta_0] \to \gamma (\delta) \in B_{\rho_0}(x^*), 
$$
where $B_{\rho_0}(x^*)$ is the ball centered at $x^*$ with radius $\rho_0 >0$ (to be fixed).
We endow $\Gamma_{\rho_0, \delta_0}(x^*, 0)$ with the distance 
$$\text{dist} (\gamma, \gamma' ) = \sup_{|\delta| < \delta_0} | \gamma(\delta) - \gamma'(\delta) |;$$
this will give us s continuous curve as fixed point. We could consider a distance which takes into 
account also the derivative of $\gamma$
to show that the fixed point is a $C^1$ curve (in  fact is $C^k$ if $W \in C^k$), 
however this will be unessential for our purposes. 
We define
$$
\Phi : \gamma(\delta) \in \Gamma_{\rho_0, \delta_0}(x^*, 0) 
 \to \Phi (\gamma(\delta)) := \gamma(\delta) - \left( \nabla_x F(x^*, \delta) \right)^{-1} F(\gamma(\delta), \delta).
$$
We will show that $\Phi$ is a contraction, so that its fixed point $x(\delta)$ satisfies 
$F(x(\delta), \delta) =0$, as claimed.
We define (with a small abuse of notation) for $x \in  B_{\rho_0}(x^*)$
$$
\Phi(x, \delta) = x - \left( \nabla_x F(x^*, \delta) \right)^{-1} F(x, \delta).
$$  
Note that
$$
\nabla_x \Phi(x, \delta) = \mbox{Id} - \left( \nabla_x F(x^*, \delta) \right)^{-1}  \nabla_x F(x, \delta)
$$
Moreover, expanding the sine and cosine we get 
\begin{equation}\label{BeforeLHalf}
\nabla_x F(x, \delta) = \pm \begin{pmatrix}
nm + \mathcal{O}(N^{4-L} ) & \mathcal{O}(N^{2-2L}) + \mathcal{O}(\delta) \\
\mathcal{O}(N^{2-2L}) + \mathcal{O}(\delta)  & nm + \mathcal{O}(N^{4-L} ) 
 \end{pmatrix}
 \end{equation}
for all $x$ such that $|x-x^*| < C N^{-L}$, namely taking 
\begin{equation}
\rho_0 = C N^{-L},
\end{equation} 
where the constant $C$ that depends only on $\|W\|_{C^1}$ will be chosen later.
Note that this choice of $\rho_0$ 
automatically implies that the new critical points that we found do not collapse into each other, 
(taking for instance $L\geq 2$) for all $N$ sufficiently large (depending on $L, \|W\|_{C^1}$.)
Thus the perturbed vector field has more that $8nm$ (regular) critical points.   

Now, fixing $\delta_0(N) = N^{-L}$ and taking $L$ sufficiently large, we rewrite 
\eqref{BeforeLHalf} as
\begin{equation}\label{FinalBeforeInv}
\nabla_x F(x, \delta) = \pm \begin{pmatrix}
nm + \mathcal{O}(N^{-L/2} ) & \mathcal{O}(N^{-L/2})  \\
\mathcal{O}(N^{-L/2}) +   & nm + \mathcal{O}(N^{-L/2} ) 
 \end{pmatrix},
\end{equation}
and we can compute 
\begin{equation}\label{LipCostContr}
\nabla_x \Phi(x, \delta) = \begin{pmatrix}
1 + \Omega^{-1} ( n^2m^2 + \mathcal{O}(N^{-L/3}))  & \mathcal{O}(N^{-L/3}) \\
\mathcal{O}(N^{-L/3}) & 1 + \Omega^{-1} (  n^2m^2 + \mathcal{O}(N^{-L/3}) )  
 \end{pmatrix},
 \end{equation}
where 
and 
$$
\Omega:=  - n^2 m ^2 + \mathcal{O}(N^{-L/3}) .
$$
Then, noting that 
$$
\frac{n^2m^2}{\Omega} = - 1 + \mathcal{O}(N^{-L/3}),
$$
we see that, for all 
$N$ sufficiently large 
\begin{equation}\label{ContrHp1}
| \nabla_x \Phi(x, \delta) | \ll 1.
\end{equation}
This implies that $\Phi$ contracts the distances. To show that 
$\Phi$ is a contraction it now suffices to prove~\eqref{BallToItself} below, that we will deduce by 
\begin{equation}\label{ContrHp2}
|  \Phi(x^*, \delta) - x^*   | \ll \rho_0;
\end{equation}
namely the constant curve 
$\gamma(t) : [-\delta_0, \delta_0] \to x^*$ does not change too much under $\Phi$. Indeed, 
combining~\eqref{ContrHp1} and~\eqref{ContrHp2} and using the triangle inequality, we readily see 
that \begin{equation}\label{BallToItself}
\Phi \left( \Gamma_{\rho_0, \delta_0}(x^*, 0) \right) \subset \Gamma_{\rho_0, \delta_0}(x^*, 0).
\end{equation}

On the other hand  
$$
\Phi(x^*, \delta) - x^* = - \left( \nabla_x F(x^*, \delta) \right)^{-1} F(x^*, \delta),
$$
and 
$$
F(x^*,\delta)= \delta W(x^*);
$$
recall \eqref{DefFPerturbation} and $V_N(x^*) =0$.
Thus we can use \eqref{FinalBeforeInv} to estimate  
\begin{align}\nonumber
|  \Phi(x^*, \delta) - x^*  | & = 
 \left| \frac{1}{\Omega} 
 \begin{pmatrix}
nm + \mathcal{O}(N^{-L/2} ) & \mathcal{O}(N^{-L/2})  \\
\mathcal{O}(N^{-L/2}) +   & nm + \mathcal{O}(N^{-L/2} ) 
 \end{pmatrix} \begin{pmatrix}
\delta W_1(x)   \\
\delta W_2(x)
 \end{pmatrix} 
 \right| 
\\ \nonumber
& \qquad \qquad \lesssim  |\delta| \| W \|_{C^1} \leq |\delta| \| W \|_{C^1} \ll \rho_0,
\end{align}
which lead to \eqref{ContrHp2} for all $N$ sufficiently large. In the last inequality we used that 
$\delta_0 = N^{-L}$ and $\rho_0 = C N^{-L}$ and we have taken the constant $C$ sufficiently large 
(in fact a large multiple of $\| W \|_{C^1}$).
This conclude the argument for the hyperbolic critical points.

For the elliptic critical points we proceed in the same way. Here we sketch the relevant calculations. First of all, we have that
\begin{equation}
\nabla F(\bar{x},\delta)= \pm \begin{pmatrix}
0 & - m^2 \\
 n^2  & 0
\end{pmatrix} + \delta W,
\end{equation}
and, as above, we get
thus 
\begin{equation}\label{FirstMathcalO}
\det \nabla_x F(\bar{x}, \delta)  = 
 - n^2 m ^2 + \mathcal{O}(\delta nm) ,
\end{equation}
and so 
$$
 | \det \nabla_x F(\bar{x}, \delta) | > \frac{N^2}{2}.
$$
Defining $\Phi$ as 
$$
\Phi(x, \delta) = x - \left( \nabla_x F(\bar{x}, \delta) \right)^{-1} F(x, \delta),
$$
we can show that $\nabla_x \Phi(x, \delta)$ has also the form \eqref{LipCostContr} and that it contract the 
distances. Then essentially the same calculations as above 
allows us to prove the analogous of \eqref{ContrHp2}, with $x^*$ replaced by $\bar{x}$, and to conclude with the same computations as in the previous case.
\end{proof}

\section{Stability of the MHD system in 2D}\label{Sec:2dStab}
In this section we prove some preliminary results on the system \eqref{eq:mhd} in the two-dimensional setting.

\subsection{Boundedness of the $H^r$ norms} We start by proving an a priori estimate for the the Sobolev norms of the solution. This result (Theorem \ref{thm:boundhr}) will provide us some useful estimates in order to prove Theorem \ref{Lem:HrDecay}, which quantifies the decay of the velocity as the viscosity becomes large. Moreover, a perturbative version of Theorem \ref{thm:boundhr}, namely Theorem \ref{thm:boundhrStab}, will be used directly in the proof of the 2D magnetic reconnection in Theorem \ref{thm:main}.

\begin{rem}
While the propagation of the $H^r$ regularity for 2D solutions is not surprising, the relevant information here is the fact that we control the growth exponentially in the~$L^2$ norms while only polynomially in the higher order Sobolev norms. 
\end{rem}

Define 
$
\Gamma := 1 + \|u_0\|_{L^2}^2+\|b_0\|_{L^2}^2
$
and take $\sigma < \min ( \nu,\eta )$. 

\begin{thm}\label{thm:boundhr}
Let $u_0,b_0\in H^r(\T^2)$ be two divergence-free vector fields with zero mean. Assume $$\| u_0 \|_{H^m} + \| b_0 \|_{H^m} \leq C N^m, \qquad m=0, \ldots, r,$$ for some $N > 1$. Let $(u,b,p)$ be the unique solution \eqref{eq:mhd} with initial datum $(u_0,b_0)$, then for all~$\kappa < 2$
\begin{equation}\label{est:HrWD}
\|u(t,\cdot)\|_{H^r}^2+\|b(t,\cdot)\|_{H^r}^2 + \kappa \nu \int_0^t\|\nabla^{r+1}u(s,\cdot)\|_{L^2}^2\de s
+ \kappa \eta \int_0^t\|\nabla^{r+1}b(s,\cdot)\|_{L^2}^2\de s\leq  
N^{2r}
 e^{\frac{C\Gamma}{\sigma^2}},
\end{equation}
and
\begin{equation}\label{est:Hr}
\|u(t,\cdot)\|_{H^r}^2+\|b(t,\cdot)\|_{H^r}^2
\leq 
 N^{2r}  e^{-2\sigma t} e^{\frac{C \Gamma}{\sigma^2}}.
\end{equation}
where the implicit constants $C$ depend on $\sigma, \kappa, r$.
\end{thm}

When $r=0$ we have of course more precise estimates than \eqref{est:HrWD}, see for instance estimate \eqref{ImprovedR=0} and the energy estimate \eqref{EnergyLast?}. A similar comment applies at least to small values of $r$, however the general form \eqref{est:Hr}, which is enough for our purposes, is well suited to be generalized to the corresponding stability estimate (see Theorem \ref{thm:boundhrStab}). 

\begin{proof}
We divide the proof in several steps.\\
\\
\underline{\em Step 1} \,\,Estimate for $r=0$.\\
\\
By standard arguments, we know that the following identity holds for smooth solutions
\begin{equation}\label{eq:energy-identity}
\frac{1}{2}\frac{\de}{\de t}\int_{\T^2}\left(|u(t,x)|^2+|b(t,x)|^2\right)\de x+\nu\int_{\T^2}|\nabla u(t,x)|^2\de x+\eta \int_{\T^2}|\nabla b(t,x)|^2\de x=0.
\end{equation}
By integrating the above identity in time we obtain that
\begin{equation}\label{EnergyLast?}
\|u(t,\cdot)\|_{L^2}^2+\|b(t,\cdot)\|_{L^2}^2+2\nu\int_0^t\|\nabla u(s,\cdot)\|^2_{L^2}\de s+2\eta\int_0^t\|\nabla b(s,\cdot)\|^2_{L^2}\de s = \|u_0\|_{L^2}^2+\|b_0\|_{L^2}^2,
\end{equation}
which gives the \eqref{est:HrWD} for $r=0$. Moreover, note that if we apply by Poincar\'e's inequality to \eqref{eq:energy-identity} we obtain that
\begin{equation}
\frac{\de}{\de t}\left(\|u(t,\cdot)\|_{L^2}^2+\|b(t,\cdot)\|_{L^2}^2\right)\leq -2\nu\|u(t,\cdot)\|_{L^2}^2-2\eta \|b(t,\cdot)\|_{L^2}^2,
\end{equation}
and by Gronwall's inequality (recall $\sigma < \min (\eta, \nu)$)
\begin{equation}\label{ImprovedR=0}
\|u(t,\cdot)\|_{L^2}^2+\|b(t,\cdot)\|_{L^2}^2\leq \left(\|u_0\|_{L^2}^2+\|b_0\|_{L^2}^2\right) e^{-2\sigma t}.
\end{equation}
\\
\underline{\em Step 2} \,\,Estimate for $r=1$.\\
\\
We differentiate the equations and multiplying the first by $\nabla u$ and the second by $\nabla b$ we obtain that
\begin{align*}
\frac{1}{2}\frac{\de}{\de t}(\|\nabla u(t,\cdot)\|_{L^2}^2&+\|\nabla b(t,\cdot)\|_{L^2}^2)+\nu \|\nabla^2 u(t,\cdot)\|_{L^2}^2+\eta \|\nabla^2 b(t,\cdot)\|_{L^2}^2+\int_{\T^2}\nabla [(u\cdot\nabla)u]:\nabla u\de x\\
&+\int_{\T^2}\nabla [(u\cdot\nabla)b]:\nabla b\de x=\int_{\T^2}\nabla [(b\cdot\nabla)b]:\nabla u\de x+\int_{\T^2}\nabla [(b\cdot\nabla)u]:\nabla b\de x.
\end{align*}
Thanks to the properties of the transport operator we have several cancellations which yield to
\begin{equation}
\frac{1}{2}\frac{\de}{\de t}(\|\nabla u(t,\cdot)\|_{L^2}^2+\|\nabla b(t,\cdot)\|_{L^2}^2)+\nu \|\nabla^2 u(t,\cdot)\|_{L^2}^2+\eta \|\nabla^2 b(t,\cdot)\|_{L^2}^2\leq \int_{\T^2}|\nabla u||\nabla u|^2\de x+3\int_{\T^2}|\nabla u||\nabla b|^2\de x
\end{equation}
We now use Holder, Ladyzenskaya, and Young inequality to get that
\begin{align*}
\int_{\T^2}|\nabla u||\nabla u|^2\de x&\leq \|\nabla u(t,\cdot)\|_{L^2}\|\nabla u(t,\cdot)\|_{L^4}^2\\
&\leq \|\nabla u(t,\cdot)\|_{L^2}^2 \|\nabla^2 u(t,\cdot)\|_{L^2}\\
&\leq \frac{C}{ \nu \varepsilon} \|\nabla u(t,\cdot)\|_{L^2}^4+ \nu \varepsilon \|\nabla^2 u(t,\cdot)\|_{L^2}^2,
\end{align*}
and
\begin{align*}
3\int_{\T^2}|\nabla u||\nabla b|^2\de x&\leq \|\nabla u(t,\cdot)\|_{L^2}\|\nabla b(t,\cdot)\|_{L^4}^2\\
&\leq 3\|\nabla u(t,\cdot)\|_{L^2}\|\nabla b(t,\cdot)\|_{L^2} \|\nabla^2 b(t,\cdot)\|_{L^2}\\
&\leq \frac{C}{\eta \varepsilon} \|\nabla u(t,\cdot)\|_{L^2}^2\|\nabla b(t,\cdot)\|_{L^2}^2 + \eta \varepsilon \|\nabla^2 b(t,\cdot)\|_{L^2}^2.
\end{align*}
Then, by ``absorbing constants" and taking $\varepsilon$ small we get, for all $\sigma < \min (\eta, \nu)$
and $\kappa < 2$
\begin{align}\label{wiitt}
\frac{1}{2}\frac{\de}{\de t}(\|\nabla u(t,\cdot)\|_{L^2}^2+&\|\nabla b(t,\cdot)\|_{L^2}^2)+  \kappa \nu \|\nabla^2 u(t,\cdot)\|_{L^2}^2
+ \kappa \eta  \|\nabla^2 b(t,\cdot)\|_{L^2}^2\\
&\nonumber 
\leq \frac{C}{\sigma} \|\nabla u(t,\cdot)\|_{L^2}^4+ \frac{C}{\sigma} \|\nabla u(t,\cdot)\|_{L^2}^2\|\nabla b(t,\cdot)\|_{L^2}^2\\
&\nonumber
\leq \frac{C}{\sigma}\|\nabla u(t,\cdot)\|_{L^2}^2\left(\|\nabla u(t,\cdot)\|_{L^2}^2+\|\nabla b(t,\cdot)\|_{L^2}^2\right),
\end{align}
where $C$ depends on  $\kappa$ and $\sigma$. Finally, by Poincar\'e's inequality we obtain
\begin{equation}\label{eq:H1-poinc}
\frac{\de}{\de t}(\|\nabla u(t,\cdot)\|_{L^2}^2+\|\nabla b(t,\cdot)\|_{L^2}^2)
\leq \left(-2 \sigma + \frac{C}{\sigma}\|\nabla u(t,\cdot)\|_{L^2}^2\right)\left(\|\nabla u(t,\cdot)\|_{L^2}^2+\|\nabla b(t,\cdot)\|_{L^2}^2\right),
\end{equation}
and then, by Gronwall's lemma and the energy inequality we obtain that
\begin{equation}
\|\nabla u(t,\cdot)\|_{L^2}^2+\|\nabla b(t,\cdot)\|_{L^2}^2\leq \left(\|\nabla u_0\|_{L^2}^2+\|\nabla b_0\|_{L^2}^2\right)e^{\frac{C \Gamma}{\sigma^2}}e^{-2\sigma t}.
\end{equation}
Finally, we integrate in time \eqref{wiitt} and by using the energy inequality and the estimate above we obtain that
\begin{equation}
\kappa \nu \int_0^t \|\nabla^2 u(s,\cdot)\|_{L^2}^2\de s+ \kappa \eta \int_0^t \|\nabla^2 b(s,\cdot)\|_{L^2}^2\de s
\leq \left(\|\nabla u_0\|_{L^2}^2+\|\nabla b_0\|_{L^2}^2\right)\left(1+\frac{CC_0}{\sigma^2}e^{\frac{C \Gamma}{\sigma^2}}\right),
\end{equation}
which implies \eqref{est:HrWD}.\\
\\
\underline{\em Step 3} \,\,Inductive step.\\
\\
Let $r\geq 2$ and assume that the estimate \eqref{est:Hr} holds for $r-1$. Let $\alpha\in\N^2$ with $|\alpha|=r$, we differentiate the equation by $\partial^\alpha$ and we multiply by $\partial^\alpha u$ and $\partial^\alpha b$ respectively the equations for $u$ and $ b$ in \eqref{eq:mhd}, obtaining that
\begin{align*}
\frac{1}{2}\frac{\de}{\de t} &(\|\partial^\alpha u(t,\cdot)\|_{L^2}^2 +\|\partial^\alpha b(t,\cdot)\|_{L^2}^2)+\nu \|\nabla\partial^\alpha u(t,\cdot)\|_{L^2}^2 +\eta \|\nabla\partial^\alpha b(t,\cdot)\|_{L^2}^2\\
&+ \int_{\T^2}\partial^\alpha [(u\cdot\nabla)u]:\partial^\alpha u\de x+\int_{\T^2}\partial^\alpha [(u\cdot\nabla)b]:\partial^\alpha b\de x=\int_{\T^2}\partial^\alpha [(b\cdot\nabla)u]:\partial^\alpha b\de x\\
&+\int_{\T^2}\partial^\alpha [(b\cdot\nabla)b]:\partial^\alpha u\de x.
\end{align*}
Then, by using again the properties of the transport operator we have that
\begin{align}
\frac{1}{2}\frac{\de}{\de t} &(\|\partial^\alpha u(t,\cdot)\|_{L^2}^2 +\|\partial^\alpha b(t,\cdot)\|_{L^2}^2)+\nu \|\nabla\partial^\alpha u(t,\cdot)\|_{L^2}^2 +\eta \|\nabla\partial^\alpha b(t,\cdot)\|_{L^2}^2\nonumber\\
&+\sum_{\beta\leq \alpha,\beta\neq 0} \int_{\T^2}\partial^\beta u\cdot\nabla \partial^{\alpha-\beta}u:\partial^\alpha u\de x+ \sum_{\beta\leq \alpha,\beta\neq 0} \int_{\T^2}\partial^\beta u\cdot\nabla \partial^{\alpha-\beta}b:\partial^\alpha b\de x\nonumber\\
&=\sum_{\beta\leq \alpha,\beta\neq 0} \int_{\T^2}\partial^\beta b\cdot\nabla \partial^{\alpha-\beta}u:\partial^\alpha b\de x+\sum_{\beta\leq \alpha,\beta\neq 0} \int_{\T^2}\partial^\beta b\cdot\nabla \partial^{\alpha-\beta}b:\partial^\alpha u\de x.\label{eq:identity-a-Hr}
\end{align}
By summing over $|\alpha|=r$ and $0\neq \beta\leq \alpha$ we get the estimate
\begin{align}
\frac{1}{2}\frac{\de}{\de t} (\|\nabla^r u(t,\cdot)\|_{L^2}^2 & +\|\nabla^r b(t,\cdot)\|_{L^2}^2)+\nu \|\nabla^{r+1} u(t,\cdot)\|_{L^2}^2 +\eta \|\nabla^{r+1} b(t,\cdot)\|_{L^2}^2\nonumber\\
&\leq \sum_{m=1}^r\left(\int_{\T^2}|\nabla^m u||\nabla^{r+1-m}u||\nabla^r u|\de x+ \int_{\T^2}|\nabla^m u||\nabla^{r+1-m} b||\nabla^r b|\de x\right)\nonumber\\
&+\sum_{m=1}^r\left(\int_{\T^2}|\nabla^m b||\nabla^{r+1-m}u||\nabla^r b|\de x+ \int_{\T^2}|\nabla^m b||\nabla^{r+1-m} b||\nabla^r u|\de x\right).\label{est:somma-hr}
\end{align}
Now note that the second and the third terms on the right hand side are the same. Then, consider the terms of the sum with $m=1,r$: we have the following estimates
\begin{align}\label{like2}
\int_{\T^2}|\nabla u(t,x)||\nabla^r u(t,x)|^2\de x & \leq \|\nabla u(t,\cdot)\|_{L^2}\|\nabla^r u(t,\cdot)\|^2_{L^4} \nonumber \\
&\leq \|\nabla u(t,\cdot)\|_{L^2}\|\nabla^r u(t,\cdot)\|_{L^2}\|\nabla^{r+1} u(t,\cdot)\|_{L^2} \nonumber \\
&\leq \frac{C}{\nu \varepsilon}\|\nabla u(t,\cdot)\|_{L^2}^2\|\nabla^r u(t,\cdot)\|_{L^2}^2+ \nu \varepsilon \|\nabla^{r+1} u(t,\cdot)\|_{L^2}^2.
\end{align}

\begin{align}\label{like3}
\int_{\T^2}|\nabla u(t,x)||\nabla^r b(t,x)|^2\de x & \leq \|\nabla u(t,\cdot)\|_{L^2}\|\nabla^r b(t,\cdot)\|^2_{L^4} \nonumber \\
&\leq \|\nabla u(t,\cdot)\|_{L^2}\|\nabla^r b(t,\cdot)\|_{L^2}\|\nabla^{r+1} b(t,\cdot)\|_{L^2} \nonumber \\
&\leq \frac{C}{\eta \varepsilon}\|\nabla u(t,\cdot)\|_{L^2}^2\|\nabla^r b(t,\cdot)\|_{L^2}^2+ \eta \varepsilon \|\nabla^{r+1} b(t,\cdot)\|_{L^2}^2.
\end{align}

\begin{align}
\int_{\T^2}|\nabla b(t,x)||\nabla^r b(t,x)||\nabla^r u(t,x)|\de x &\leq \int_{\T^2}|\nabla b(t,x)||\nabla^r b(t,x)|^2\de x+\int_{\T^2}|\nabla b(t,x)||\nabla^r u(t,x)|^2\de x\nonumber\\
&\leq \frac{C}{\sigma \varepsilon}\|\nabla b(t,\cdot)\|_{L^2}^2(\|\nabla^r u(t,\cdot)\|_{L^2}^2+\|\nabla^r b(t,\cdot)\|_{L^2}^2)\nonumber\\
&+ \nu \varepsilon \|\nabla^{r+1} u(t,\cdot)\|_{L^2}^2+ \eta \varepsilon \|\nabla^{r+1} b(t,\cdot)\|_{L^2}^2.
\end{align}

If $r=2$ this is enough to conclude since, taking $\varepsilon$ small we have obtained that
\begin{align*}
\frac{\de}{\de t} (\|\nabla^r u(t,\cdot)\|_{L^2}^2 & +\|\nabla^r b(t,\cdot)\|_{L^2}^2) + \kappa \nu  \|\nabla^{r+1} u(t,\cdot)\|_{L^2}^2 +\kappa \eta \|\nabla^{r+1} b(t,\cdot)\|_{L^2}^2\\
&\leq \frac{C}{\sigma}\left(\|\nabla u(t,\cdot)\|_{L^2}^2+\|\nabla b(t,\cdot)\|_{L^2}^2\right)(\|\nabla^r u(t,\cdot)\|_{L^2}^2+\|\nabla^r b(t,\cdot)\|_{L^2}^2),
\end{align*}
and as we did several times before, by Poincar\'e and Gronwall inequality we get that
\begin{equation}
\|\nabla^r u(t,\cdot)\|_{L^2}^2 +\|\nabla^r b(t,\cdot)\|_{L^2}^2\leq \left(\|\nabla^r u_0\|_{L^2}^2 +\|\nabla^r b_0\|_{L^2}^2\right)e^{\frac{C\Gamma}{\sigma^2}}e^{-2\sigma t},
\end{equation}
which also implies after integration in time
\begin{align}
\kappa \nu \int_0^t\|\nabla^{r+1} u(s,\cdot)\|_{L^2}^2 \de s&+ \kappa \eta \int_0^t \|\nabla^{r+1} b(s,\cdot)\|_{L^2}^2\de s\leq \|\nabla^r u_0\|_{L^2}^2 +\|\nabla^r b_0\|_{L^2}^2\nonumber\\
&+\frac{C}{\sigma}\int_0^t\left(\|\nabla u(s,\cdot)\|_{L^2}^2+\|\nabla b(s,\cdot)\|_{L^2}^2\right)(\|\nabla^r u(s,\cdot)\|_{L^2}^2+\|\nabla^r b(s,\cdot)\|_{L^2}^2)\de s \nonumber\\
&\leq 
(\|u_0\|_{H^r}^2+\|b_0\|_{H^r}^2)
 e^{\frac{C\Gamma}{\sigma^2}}.
\end{align}

Then, we assume that $r\geq 3$ and we estimate the remaining terms in \eqref{est:somma-hr} as follows: back to \eqref{eq:identity-a-Hr}, we exploit the divergence-free condition and by integrating by parts we have to bound the following integrals
\begin{align}
\sum_{m=2}^{r-1}\int_{\T^2} |\nabla^m u(t,x)||\nabla^{r-m} u(t,x)||\nabla^{r+1} u(t,x)|\de x&\leq  \sum_{m=2}^{r-1}C\int_{\T^2} |\nabla^m u(t,x)|^2|\nabla^{r-m} u(t,x)|^2\de x\nonumber\\
&+\frac{\nu}{c}\int_{\T^2}|\nabla^{r+1} u(t,x)|^2 \de x,\nonumber
\end{align}
and then by Ladyzenskaya
\begin{align}
\int_{\T^2} |\nabla^m u(t,x)|^2&|\nabla^{r-m} u(t,x)|^2\de x\leq\|\nabla^m u(t,\cdot)\|_{L^4}^2\|\nabla^{r-m} u(t,\cdot)\|_{L^4}^2\nonumber\\
&\leq \|\nabla^m u(t,\cdot)\|_{L^2}\|\nabla^{m+1} u(t,\cdot)\|_{L^2}\|\nabla^{r-m} u(t,\cdot)\|_{L^2}\|\nabla^{r-m+1} u(t,\cdot)\|_{L^2}\nonumber\\
&\leq \frac{1}{2}\|\nabla^m u(t,\cdot)\|_{L^2}^2\|\nabla^{r-m+1} u(t,\cdot)\|_{L^2}^2 +
\frac{1}{2} \|\nabla^{m+1} u(t,\cdot)\|_{L^2}^2\|\nabla^{r-m} u(t,\cdot)\|_{L^2}^2.
\end{align} 
Note that choosing $m= r-1$ in the second term we have a contribution 
$ \|\nabla^{r} u(t,\cdot)\|_{L^2}^2\|\nabla u(t,\cdot)\|_{L^2}^2$. Thus, recalling the previous estimates for the contributions $m=1,r$, after rearranging the indexes in the sums we arrive to 
\begin{align}
\frac{\de}{\de t} (\|\nabla^r u(t,\cdot)\|_{L^2}^2 &+\|\nabla^r b(t,\cdot)\|_{L^2}^2)\leq - \kappa \nu \|\nabla^{r} u(t,\cdot)\|_{L^2}^2 -\kappa \eta \|\nabla^{r} b(t,\cdot)\|_{L^2}^2\nonumber\\
&+\frac{C}{\sigma}\left(\|\nabla u(t,\cdot)\|_{L^2}^2+\|\nabla b(t,\cdot)\|_{L^2}^2\right)(\|\nabla^r u(t,\cdot)\|_{L^2}^2+\|\nabla^r b(t,\cdot)\|_{L^2}^2)\label{like3}\\
&+\frac{C}{\sigma}\sum_{m=2}^{r-1}\left(\|\nabla^{r-m+1} u(t,\cdot)\|_{L^2}^2\|\nabla^{m} u(t,\cdot)\|_{L^2}^2 \label{like4}
\right).
\end{align}
We estimate the last group of terms using the induction assumption \eqref{est:Hr} for $m=2, \ldots, r-1$. This yelds 
\begin{equation}
 \|\nabla^{r-m+1} u(t,\cdot)\|_{L^2}^2  \|\nabla^{m} u(t,\cdot)\|_{L^2}^2 
\leq 
N^{2(r-m+1)}  e^{\frac{C \Gamma}{\sigma^2}}
  \|\nabla^{m} u(t,\cdot)\|_{L^2}^2 .
\end{equation}
Plugging this into \eqref{like4} we arrive to 
\begin{align}
\frac{\de}{\de t} (\|\nabla^r u(t,\cdot)\|_{L^2}^2 &+\|\nabla^r b(t,\cdot)\|_{L^2}^2)\leq - \kappa \nu \|\nabla^{r} u(t,\cdot)\|_{L^2}^2 -\kappa \eta \|\nabla^{r} b(t,\cdot)\|_{L^2}^2\nonumber\\
&+\frac{C}{\sigma}\left(\|\nabla u(t,\cdot)\|_{L^2}^2+\|\nabla b(t,\cdot)\|_{L^2}^2\right)(\|\nabla^r u(t,\cdot)\|_{L^2}^2+\|\nabla^r b(t,\cdot)\|_{L^2}^2)\nonumber\\
&+ \frac{C}{\sigma}\sum_{m=2}^{r-1} N^{2(r-m+1)}  e^{\frac{C \Gamma}{\sigma^2}}
  \|\nabla^{m} u(t,\cdot)\|_{L^2}^2. \label{like5}
\end{align}
Using Poincar\'e inequality and the Gronwall lemma we obtain
the desired estimate \eqref{est:Hr}. Note that the time integral of the last contribution of the right hand side of \eqref{like5} is estimated using the induction assumption \eqref{est:HrWD}. Integrating \eqref{like5} we also obtain the \eqref{est:HrWD} and the proof is concluded.
\end{proof}

\subsection{Boundedness of the $H^r$ norms of the difference equation \eqref{eq:differenceTris}}
We also need a perturbative version of Theorem \ref{thm:boundhr}, in order to control the $H^r$ norms of solutions of the difference system 
\begin{equation}\label{eq:differenceTris}
\begin{cases}
\partial_t v+\dive\left(v\otimes v+2v\otimes u\right)+\nabla P_{v,h}=\nu\Delta v+\dive\left(h\otimes h+2h\otimes b\right),\\
\partial_t h+(v\cdot\nabla )h+(u\cdot \nabla)h+(v\cdot\nabla)b=\eta\Delta h+(h\cdot\nabla)v+(h\cdot\nabla)u+(b\cdot\nabla)v, \\
\dive v=\dive h=0,
\end{cases}
\end{equation}
where $(u,b)$ is a solution of \eqref{eq:mhd}. Note that $w=u+v$, $m=b+h$ solves the \eqref{eq:mhd} equation (with an appropriate choice of the pressure).

To do so 
we define 
$$
\tilde \Gamma :=  1 +
\|u_0\|_{L^2}^2+\|b_0\|_{L^2}^2 
+ \|v_0\|_{L^2}^2+\|h_0\|_{L^2}^2
$$
and we take $\sigma < \min ( \nu,\eta )$.

\begin{thm}\label{thm:boundhrStab}
Let $v_0,h_0\in H^r(\T^2)$ be two divergence-free vector fields with zero mean. Let $(v,h,P_{v,h})$ be the unique solution \eqref{eq:differenceTris} with initial datum $(v_0,h_0)$, where $(u,b)$ is a solution of \eqref{eq:mhd}. Assume that 
$$
\| u_0 \|_{H^m} + \| b_0 \|_{H^m}  \leq C N^m, \qquad \| v_0 \|_{H^m} + \| h_0 \|_{H^m} \leq C \delta, \qquad m=0, \ldots, r,
$$  
for some $N > 1$.
Then
\begin{equation}\label{est:HrStab}
\|h(t,\cdot)\|_{H^r}^2 + \|v(t,\cdot)\|_{H^r}^2
\leq  
\delta^2 N^{2r} e^{-2\sigma t} e^{\frac{C \tilde\Gamma}{\sigma^2}},
\end{equation}
where the implicit constants $C$ depend on $r, \sigma$.
\end{thm}

The proof is a straightforward generalization of that of Theorem \ref{thm:boundhr}, we left the details to the reader. Again, at least for small values of $r$ we may prove stronger estimates. For instance for $r=0$ we may set $\tilde \Gamma :=  \|u_0\|_{L^2}^2+\|b_0\|_{L^2}^2$. The bound \eqref{est:HrStab} is however sufficient for our purposes.

\subsection{Decay of the velocity}
We conclude the section proving that the $H^r$ norm of the velocity decays in $\nu$ (namely for large viscosity). This result (Theorem \ref{Lem:HrDecay}) will be useful to prove 2D magnetic reconnection with arbitrary initial velocity and large viscosity, namely Theorem \ref{thm:main2}. 

%

\begin{thm}\label{Lem:HrDecay}
Let $u_0,b_0\in H^r(\T^2)$ be two divergence-free vector fields with zero mean. Assume $$\| u_0 \|_{H^m} + \| b_0 \|_{H^m} \leq C N^m, \qquad m=0, \ldots, r,$$ for some $N > 1$. Let $(u,b,p)$ be the unique solution of \eqref{eq:mhd} with initial datum $(u_0,b_0)$. Assume that $\nu>3\eta$, then
\begin{equation}\label{DecLargeNu}
\|\nabla^r u(\cdot, t)\|_{L^2}^2 \leq \left(\|\nabla^r u_0\|_{L^2}^2+\frac{Ce^{\nu t}e^{-\eta t}}{\nu \eta}N^{2r}
\right) e^{\frac{C \Gamma}{\sigma^2}} e^{-\nu t} .
\end{equation}
\end{thm}

\begin{rem}
This bound will be important for two reasons. The first, is that the right hand side goes to zero as $\nu \to \infty$. The second, is that, as in all the 2D results of this paper, the estimate is exponential in the $L^2$ norm of the initial datum, while only polynomial in the higher order Sobolev norms.
\end{rem}

\begin{proof}
We divide the proof in several steps.\\
\\
\underline{\em Step 1} $L^2$ estimate\\
\\
Let us consider the equation for the velocity field, namely
$$
\partial_t u+(u\cdot \nabla)u+\nabla P=\nu\Delta u+(b\cdot \nabla)b.
$$
Multiply the equation above by $u$ and integrating by parts it is easy to show that
\begin{align*}
\frac{1}{2}\frac{\de}{\de t}\|u(t,\cdot)\|_{L^2}^2+\nu\|\nabla u(t,\cdot)\|_{L^2}^2&=-\int_{\T^2}b\otimes b:\nabla u\de x\\
&\leq \|b(t,\cdot)\|^2_{L^4}\|\nabla u(t,\cdot)\|_{L^2}\\
&\leq \|b(t,\cdot)\|_{L^2}\|\nabla b(t,\cdot)\|_{L^2}\|\nabla u(t,\cdot)\|_{L^2}\\
&\leq \frac{C}{\nu}\|b(t,\cdot)\|_{L^2}^2\|\nabla b(t,\cdot)\|_{L^2}^2+\frac{\nu}{2}\|\nabla u(t,\cdot)\|_{L^2}^2,
\end{align*}
where in the third line we used Ladyzhenskaya's inequality, while in the fourth we used Young's inequality. Then, by Poincar\'e's inequality we obtain that
\begin{equation}
\frac{\de}{\de t}\|u(t,\cdot)\|_{L^2}^2\leq -\nu\|u(t,\cdot)\|_{L^2}^2+\frac{C}{\nu}\|b(t,\cdot)\|_{L^2}^2\|\nabla b(t,\cdot)\|_{L^2}^2.
\end{equation}
Now, we use Gronwall's inequality and Theorem \ref{thm:boundhr} 
with $r=0,1$ (and say $\kappa =1$) to compute  
$$
\|u(t,\cdot)\|_{H^r}^2+\|b(t,\cdot)\|_{H^r}^2
\leq 
 N^{2r}  e^{-2\sigma t} e^{\frac{C \Gamma}{\sigma^2}}.
$$
\begin{align*}
\|u(t,\cdot)\|_{L^2}^2&\leq\left(\|u_0\|_{L^2}^2+\frac{C}{\nu}\int_0^t \|b(s,\cdot)\|_{L^2}^2\|\nabla b(s,\cdot)\|_{L^2}^2 e^{\nu s}\de s\right)e^{-\nu t}\\
&\leq \left(\|u_0\|_{L^2}^2+\frac{Ce^{\frac{C \Gamma}{\sigma^2}}}{\nu}\int_0^t \|\nabla b(s,\cdot)\|_{L^2}^2 e^{(\nu-2\sigma) s}\de s\right)e^{-\nu t}\\
&\leq \left(\|u_0\|_{L^2}^2+\frac{Ce^{\frac{C \Gamma}{\sigma^2}}}{\nu}e^{(\nu-2\sigma) t}\int_0^t \|\nabla b(s,\cdot)\|_{L^2}^2\de s\right)e^{-\nu t}\\
&\leq \|u_0\|_{L^2}^2 e^{-\nu t}+\frac{Ce^{\frac{C \Gamma}{\sigma^2}} }{\nu\eta}e^{- \eta t},
\end{align*}
where we have taken $\sigma > \frac{\eta}2$; this is possible as we assumed that $\eta = \min (\eta, \nu)$ 
(in fact $\nu>3\eta$).\\
\\
\underline{\em Step 2} $H^1$ estimate.\\
\\
We differentiate the equation and multiply by $\nabla u$ to get, after integration by parts
\begin{align*}
\frac{1}{2}\frac{\de}{\de t}\|\nabla u(t,\cdot)\|_{L^2}^2&+\nu\|\nabla^2 u(t,\cdot)\|_{L^2}^2\leq \int_{\T^2}|\nabla u||\nabla u|^2 \de x+ \int_{\T^2}|\nabla(b\otimes b)||\nabla^2 u|\de x\\
&\leq\frac{\nu}{2} \|\nabla^2 u(t,\cdot)\|_{L^2}^2+\frac{C}{\nu}\|\nabla u(t,\cdot)\|_{L^2}^2\|\nabla u(t,\cdot)\|_{L^2}^2+\frac{C}{\nu} \int_{\T^2}|b|^2|\nabla b|^2\de x\\
&\leq\frac{\nu}{2} \|\nabla^2 u(t,\cdot)\|_{L^2}^2+\frac{C}{\nu}\|\nabla u(t,\cdot)\|_{L^2}^2\|\nabla u(t,\cdot)\|_{L^2}^2\\
&+\frac{C}{\nu}\|b\|_{L^2}^2\|\nabla b\|_{L^2}^2+\frac{C}{\nu}\|\nabla b\|_{L^2}^2\|\nabla^2 b\|_{L^2}^2.
\end{align*}
We rewrite the inequality above by using Poincar\'e's inequality
\begin{align}
\frac{\de}{\de t}\|\nabla u(t,\cdot)\|_{L^2}^2&\leq \left(\frac{C}{\nu}\|\nabla u(t,\cdot)\|_{L^2}^2-\nu\right)\|\nabla u(t,\cdot)\|_{L^2}^2\nonumber\\
&+\frac{C}{\nu}\left(\|b(t,\cdot)\|^2_{L^2}\|\nabla b(t,\cdot)\|^2_{L^2}+\|\nabla b\|_{L^2}^2\|\nabla^2 b\|_{L^2}^2\right),
\end{align}
and as usual, Gronwall's lemma gives that
\begin{align*}
\|\nabla u(t,\cdot)\|_{L^2}^2&\leq \left(\|\nabla u_0\|_{L^2}^2+\frac{C}{\nu}\int_0^t \left(\|b(s,\cdot)\|^2_{L^2}\|\nabla b(s,\cdot)\|^2_{L^2}+\|\nabla b\|_{L^2}^2\|\nabla^2 b\|_{L^2}^2\right)e^{\nu s}\de s\right) e^{\frac{C \Gamma}{\sigma^2}} e^{-\nu t}\\
&\leq \left(\|\nabla u_0\|_{L^2}^2+\frac{C}{\nu}\int_0^t \|\nabla b(s,\cdot)\|^2_{L^2}e^{(\nu-2\sigma) s}\de s+\frac{CN^2}{\nu}\int_0^t \|\nabla^2 b(s,\cdot)\|^2_{L^2}e^{(\nu-2\sigma) s}\de s\right) e^{\frac{C \Gamma}{\sigma^2}} e^{-\nu t}\\
&\leq \left(\|\nabla u_0\|_{L^2}^2+\frac{C}{\nu}e^{(\nu-2\sigma) t}\int_0^t \|\nabla b(s,\cdot)\|^2_{L^2}\de s+\frac{CN^2}{\nu}e^{(\nu-2\sigma) t}\int_0^t \|\nabla^2 b(s,\cdot)\|^2_{L^2}\de s\right) e^{\frac{C \Gamma}{\sigma^2}} e^{-\nu t}\\
&\leq \left(\|\nabla u_0\|_{L^2}^2+\frac{C}{\nu \eta}e^{(\nu-2\sigma) t}+\frac{CN^4}{\nu\eta}e^{(\nu-2\sigma) t}\right) 
e^{\frac{C \Gamma}{\sigma^2}} e^{-\nu t},
\end{align*}
and then the conclusion follows as in the previous step.\\
\\
\underline{\em Step 3} General case.\\
\\
Let $\alpha\in \N^2$ with $|\alpha|=r$, we apply $\partial^\alpha$ to the equation and we multiply by $\partial^\alpha u$ to obtain that
\begin{align}
\frac{1}{2}\frac{\de}{\de t} \|\partial^\alpha u(t,\cdot)\|_{L^2}^2 +\nu \|\nabla\partial^\alpha u(t,\cdot)\|_{L^2}^2 &+\sum_{\beta\leq \alpha,\beta\neq 0} \int_{\T^2}\partial^\beta u\cdot\nabla \partial^{\alpha-\beta}u:\partial^\alpha u\de x\nonumber\\
&=\sum_{\beta\leq \alpha} \int_{\T^2}\partial^\beta b\cdot\nabla \partial^{\alpha-\beta}b:\partial^\alpha u\de x.
\end{align}
We analyze the third and the fourth term separately. First, we have that 
\begin{align*}
\sum_{\beta\leq \alpha,\beta\neq 0} & \int_{\T^2}\partial^\beta u\cdot\nabla \partial^{\alpha-\beta}u:\partial^\alpha u\de x
\leq 2\int_{\T^2}|\nabla u||\nabla^r u|^2\de x+\sum_{1<|\beta|<r} \int_{\T^2}\partial^\beta u\cdot\nabla \partial^{\alpha-\beta}u:\partial^\alpha u\de x\\
&\leq \frac{\nu}{4}\|\nabla^{r+1} u(t,\cdot)\|_{L^2}^2+\frac{C}{\nu}\|\nabla u(t,\cdot)\|_{L^2}^2\|\nabla^r u(t,\cdot)\|_{L^2}^2 
+\frac{C}{\nu}\sum_{m=2}^{r-1}\int_{\T^2}|\nabla^m u|^2|\nabla^{r-m} u|^2\de x \\
&\leq \frac{\nu}{4}\|\nabla^{r+1} u(t,\cdot)\|_{L^2}^2+\frac{C}{\nu}\|\nabla u(t,\cdot)\|_{L^2}^2\|\nabla^r u(t,\cdot)\|_{L^2}^2
+\frac{C}{\nu}   N^{2r} e^{\frac{C \Gamma}{\sigma^2}} e^{-4\sigma t},
\end{align*}
where we have integrated by parts in the sum $1<|\beta|<r$ and we applied Theorem \ref{thm:boundhr}. On the other hand, arguing in a similar way, for the last term we have that
\begin{align*}
\sum_{\beta\leq \alpha} \int_{\T^2}\partial^\beta b\cdot\nabla \partial^{\alpha-\beta}b:\partial^\alpha u\de x&=-\sum_{\beta\leq \alpha} \int_{\T^2}\partial^\beta b\cdot \partial^{\alpha-\beta}b:\nabla \partial^\alpha u\de x\\
&\leq \frac{\nu}{4}\int_{\T^2}|\nabla^{r+1}u|^2\de x+\frac{C}{\nu}\sum_{m=0}^r\int_{\T^2}|\nabla^m b|^2(t,\cdot)|\nabla^{r-m} b(t,\cdot)|^2\de x\\
&\leq \frac{\nu}{4}\|\nabla^{r+1} u(t,\cdot)\|_{L^2}^2+\frac{C}{\nu}\sum_{m=0}^{r}\|\nabla^m b(t,\cdot)\|_{L^4}^2\|\nabla^{r-m}b(t,\cdot)\|_{L^4}^2\\
&\leq \frac{\nu}{4}\|\nabla^{r+1} u(t,\cdot)\|_{L^2}^2+\frac{C}{2\nu}\sum_{m=0}^{r}\|\nabla^m b(t,\cdot)\|_{L^2}^2
\|\nabla^{r-m+1} b(t,\cdot)\|_{L^2}^2 \\
&+\frac{C}{2\nu}\sum_{m=0}^{r}\|\nabla^{r-m} b(t,\cdot)\|_{L^2}^2 \|\nabla^{m+1} b(t,\cdot)\|_{L^2}^2\\
&\leq \frac{\nu}{4}\|\nabla^{r+1} u(t,\cdot)\|_{L^2}^2+ \frac{C}{\nu}  e^{\frac{C \Gamma}{\sigma^2}} \sum_{m=0}^{r} N^{2(r-m)} \|\nabla^{m+1} b(t,\cdot)\|_{L^2}^2e^{-2\sigma t}.
\end{align*}
Then, by Poincar\'e's inequality we obtain that
\begin{align}
\frac{\de}{\de t} \|\nabla^r u(t,\cdot)\|_{L^2}^2 & \leq \frac{C}{\nu}(\|\nabla u(t,\cdot)\|_{L^2}^2-\nu)\|\nabla^r u(t,\cdot)\|_{L^2}^2+\frac{C}{\nu}   N^{2r} e^{\frac{C \Gamma}{\sigma^2}} e^{-4\sigma t} \nonumber\\
&+\frac{C}{\nu} \sum_{m=0}^{r} N^{2(r-m)} \|\nabla^{m+1} b(t,\cdot)\|_{L^2}^2e^{-2\sigma t}
\end{align}
and an application of Gronwall's inequality and arguing as we did in the previous steps we obtain that (again we choose $\sigma = \eta/2 $)
\begin{align*}
\|\nabla^r u(t,\cdot)\|_{L^2}^2&\leq \left(\|\nabla^r u_0\|_{L^2}^2+\frac{C}{\nu} N^{2r}\int_0^t e^{(\nu-2\eta)s}\de s \right)e^{\frac{C \Gamma}{\sigma^2}} e^{-\nu t}\\
&+\left(\frac{C}{\nu}\sum_{m=0}^{r} N^{2(r-m)}\int_0^t\|\nabla^{m+1} b(s,\cdot)\|_{L^2}^2e^{(\nu-\eta )s}\de s\right)
e^{\frac{C \Gamma}{\sigma^2}} e^{-\nu t}\\
&\leq \left(\|\nabla^r u_0\|_{L^2}^2+\frac{Ce^{\nu t}e^{-2\eta t}}{\nu(\nu-2\eta)}N^{2r}+
\sum_{m=0}^r \frac{CN^{2r}}{\nu\eta}e^{(\nu-\eta )t}\right) e^{\frac{C \Gamma}{\sigma^2}} e^{-\nu t},
\end{align*}
where we used \eqref{est:HrWD} in the last inequality. Recalling $\nu > 3 \eta$ this implies \eqref{DecLargeNu} and the proof is concluded.

\end{proof}
%

\section{Magnetic reconnection in 2D (small velocity)}\label{Sec:SmallVel}

We are ready to prove the 2D case of our main Theorem \ref{thm:main}.

\subsection{Construction of the initial datum} 
For our argument, we may chose as a reference solution any of the (large frequency) Taylor fields defined Section \ref{Sec:BeltramiTaylorBis}, however, to fix the idea we focus on
$$
\mathcal V^1_{nm}  = (m\sin nx_1 \cos mx_2, -n\cos nx_1 \sin mx_2).
$$
where $n^2+m^2=N^2$ and $N$ will be taken large. We will set, with a small abuse of notation, 
$V_N = \mathcal{V}^1_{nm}$. Also, we take $\mathcal{V}_1 \in \mathbb{V}_{1}$ and we choose, among them, one of the Hamiltonian structurally stable vector field (we know that in $\mathbb{V}_{1}$ there are some). For example, we can use 
$$
V_1=\left(\sin y,\frac{1}{2}\sin x\right).
$$

We consider the initial data
$$
m_0= b_0 + h_0, \qquad b_0 :=  \frac{M}{N} V_N, \qquad  h_0 := \delta V_1,
$$
where $M >0$ (is possibly large) and $\delta >0$ will be chosen to be very small. Indeed, the goal is to show that we can choose $\delta$ small and $N$ big enough such that if $b(t,\cdot)$ is the solution of \eqref{eq:mhd} with initial datum $b_0$, then $b_0$ will have at least $8mn$ regular critical points while at some positive time $t=T>0$ the topology of the integral lines of $b(T,\cdot)$ will be the same of $V_1$; in particular $b(T,\cdot)$ will have $4$ regular regular points. Thus a change of topology of the magnetic lines happened between $t=0$ and $t=T$.

\begin{proof}[Proof of Theorem \ref{thm:main}, d=2]
We divide the proof in several steps.\\
\\
\underline{\em Step 1} \hspace{0.3cm} Auxiliary solutions of \eqref{eq:mhd}.\\
\\
Recall that $V_N$ is divergence free
and that $\Delta V_N = -N^2 V_N$. Moreover $V_N$ solves
$$
(V_N \cdot \nabla) V_N = \nabla P_{V_N},
$$
with pressure (recall $N^2 = n^2 + m^2$) 
$$
P_{V_N} = 
 m^2 (\sin (nx_1))^2 + n^2 (\sin (m x_2))^2.
$$
Thus the couple 
$$
(u, b) = (0, e^{-\eta N^2 t} V_N ),
$$
is the unique smooth solution of \eqref{eq:mhd} with data$ (0,V_N)$; the pressure is $e^{-2\eta N^2 t} P_{V_N}$.\\
\\
%
\underline{\em Step 2} \hspace{0.3cm}Integral lines at time $0$.\\
\\
Consider the rescaled datum 
$$
\frac{N}{M} m_0 = V_N + \frac{N}{M} \delta V_1.$$
Taking 
\begin{equation}\label{VeryExplCondIP}
 \delta < c \frac{M}{N^{L+1}},
 \end{equation}
where $c$ is a suitable small constant, we have that $\frac{N}{M} \delta \ll \delta_{0}(N)$ from Lemma \ref{Lemma:CriticalPoints}, thus the vector field $\frac{N}{M} m_0$, and so $m_0$, has at least $8nm$ regular critical points.\\
\\
\underline{\em Step 3} \hspace{0.3cm}Evolution under the MHD flow.\\
\\
We denote with $(v,h)$ the solution of the difference equation \eqref{eq:difference}, with initial datum 
$$
(0, h_0), \qquad  h_0 := \delta V_1.
$$ 
Thus, under our choices, the reference solution is $(u, b) = (0, \frac{M}{N} e^{-\eta N^2 t} V_N)$ and the solution with initial datum 
$$
(0, m_0), \qquad m_0 := \frac{M}{N} V_N + h_0, \qquad  b_0 := \frac{M}{N} V_N, \quad h_0 := \delta V_1,
$$ 
is denoted by $(w, m)$.
Using the Duhamel formula we can represent 
\begin{equation}\label{MevolutionFinal}
m(t,\cdot)=\frac{M}{N} e^{-\eta N^2 t} V_N + \delta e^{-\eta t} V_1+ D(t,\cdot),
\end{equation}
where 
$$
D(t,\cdot) := L_h(t,\cdot) + L_b(t,\cdot)
$$
and we recall that $L_h, L_b$ introduced in \eqref{Def:Lh}-\eqref{Def:Lb} are defined as
\begin{equation}\nonumber
L_h(t,\cdot):=\int_0^t e^{\eta(t-s)\Delta}\dive \big(h(s)\otimes v(s)-v(s)\otimes h(s)\big)\,\de s,
\end{equation}
\begin{equation}\nonumber 
L_b(t,\cdot):=\int_0^t e^{\eta(t-s)\Delta}\dive \big(b(s)\otimes v(s)-v(s)\otimes b(s)\big)\,\de s.
\end{equation}
Since  (the constant $C$ here depends on $M$ and $r$)
$$
\| b_0 \|_{H^{m}} \leq C \| V_N\|_{H^{m}} \leq C N^{m}, \qquad
\|  h_0 \|_{H^{m}} =  \delta \|  V_1 \|_{H^{m}} \leq C \delta,   
$$
by Theorem \ref{thm:boundhrStab} we get that
\begin{align}
\|v(t,\cdot)\|_{H^m}+\|h(t,\cdot)\|_{H^m}&\leq \delta N^{m} e^{-\sigma t} e^{\frac{C \tilde\Gamma}{\sigma^2}};
\end{align}
for $m= 0, \ldots, r$.

Then, by using the above formula, we estimate the $H^r$ norms of the tensorial products in $L_h,L_b$ as follows
\begin{align}
\| h(s)\otimes v(s)\|_{H^{r+1}}&\leq \| h(s)\|_{L^\infty} \| v(s)\|_{H^{r+1}}+ \| v(s)\|_{L^\infty}\| h(s)\|_{H^{r+1}} \nonumber\\
&\leq \| h(s)\|_{H^2} \| v(s)\|_{H^{r+1}}+ \| v(s)\|_{H^2}\| h(s)\|_{H^{r+1}}\nonumber\\
&\leq C \delta^2 N^{r+3} e^{-\sigma s},\\
\| b(s)\otimes v(s)\|_{L^2}&\leq \|b(s)\|_{L^\infty}\|v(s)\|_{L^2}\leq C \delta e^{-\eta N^2 s},\\
\| b(s)\otimes v(s)\|_{H^{r+1}}&\leq C\| v(s)\|_{L^\infty} \|b(s)\|_{H^{r+1}}+ C\| v(s)\|_{H^{r+1}}\|b(s)\|_{L^\infty}\nonumber\\
&\leq C \| v(s)\|_{H^2} \|b(s)\|_{H^{r+1}}+ C\| v(s)\|_{H^{r+1}}\|b(s)\|_{L^\infty}\nonumber\\
&\leq C e^{-\eta N^2 s}\big(\delta N^{2}  +\delta  N^{r+1}\big) \nonumber \\
&\leq Ce^{-\eta N^2 s}\delta  N^{r+1}.
\end{align}
Using \eqref{UT1Fin} we estimate 
\begin{align}\label{Comb1}
\|L_h(t,\cdot)\|_{H^r}&\leq C\int_0^t \|e^{\eta(t-s)\Delta} \big(h(s)\otimes v(s)\big)\|_{H^{r+1}}\de s\nonumber\\
&\leq C\int_0^t e^{-\eta(t-s)} \|h(s)\otimes v(s)\|_{H^{r+1}}\de s\nonumber\\
& \leq  C \delta^2 N^{r+3} e^{-\sigma s} \int_0^t e^{-\eta(t-s)} e^{-\eta\sigma s}\de s\nonumber\\
&\leq C \delta^2 N^{r+3} e^{-\sigma s}.
\end{align}
Using \eqref{UTFBis} we estimate 
\begin{align}\label{Comb2}
\|L_b(t,\cdot)\|_{H^r}&\leq C\int_0^{t/2} \|e^{\eta(t-s)\Delta}\big(b(s)\otimes v(s)\big)\|_{H^{r+1}}\de s + C\int_{t/2}^t \|e^{\eta(t-s)\Delta}\big(b(s)\otimes v(s)\big)\|_{H^{r+1}}\de s\nonumber\\
&\leq C\int_0^{t/2} (t-s)^{-\frac{r+1}{2}} \|b(s)\otimes v(s)\|_{L^2}\de s + C\int_{t/2}^t e^{-\eta(t-s)} \|b(s)\otimes v(s)\|_{H^{r+1}}\de s\nonumber\\
&\leq C \delta\int_0^{t/2} (t-s)^{-\frac{r+1}{2}} e^{-\eta N^2 s}\de s + C\delta  N^2 \int_{t/2}^t e^{-\eta(t-s)}e^{-\eta N^2 s}\de s\nonumber\\
&\leq C\delta N^{-2}+ C\delta  N^{r+1} e^{- \eta N^2 \frac{t}{2}}.
\end{align}
\\
\\
\underline{\em Step 3} \hspace{0.3cm} Choice of the parameters.\\
\\
In this step we fix the parameters $N, \delta$. Recall that $\delta$ must satisfy \eqref{VeryExplCondIP}, so that $m(0,\cdot)$ has at least $8nm$ regular critical points. Then we consider the behavior of the fluid at time $t=T$. We rescale the magnetic field as
\begin{equation}
\delta^{-1}e^{\eta T} m(T,\cdot),
\end{equation}
and then recalling \eqref{MevolutionFinal} we get 
\begin{align*}
\delta^{-1} e^{\eta T} m(T,\cdot)
 =  V_1 +  \delta^{-1}  \frac{M}{N} e^{-\eta (N^2 - 1) T} V_N  + \delta^{-1} e^{\eta T} D(t,\cdot), 
\end{align*}
Our goal is to choose $N$ so large that  
\begin{equation}\label{FinalTTC}
\left\| \delta^{-1} e^{\eta T} m(T,\cdot) - V_1 \right\|_{H^r} \ll 1,
\end{equation}
for a sufficiently large\footnote{It suffices $r=3$ in order to control the norm $C^1(\T^2)$.} $r$, so that using the structural stability of $V_1$ under $C^1$ perturbations (see Section \ref{Sec:BeltramiTaylorBis}) and Sobolev embedding one can show that the set of the integral lines of $\delta^{-1} e^{\eta T} m(T,\cdot)$, and thus of $m(T,\cdot)$, is diffeomorphic to that of $V_1$. In particular $m(T,\cdot)$ has only $4$ critical points and we must have had magnetic reconnection between $t=0$ and $t=T$.

It remains to prove \eqref{FinalTTC}. We choose  
$$
\delta =  Me^{- \eta T} N^{-(L+1)}, \quad L \geq r+3,
$$  
for some sufficiently large $L$.
This is compatible with \eqref{VeryExplCondIP}
 and then combining \eqref{Comb1}-\eqref{Comb2} we get 
\begin{align}\nonumber
\| \delta^{-1} e^{\eta T} D(t,\cdot) \|_{H^r} & \leq
C \delta^{-1} e^{\eta T} ( \delta^2 N^{r+3} + \delta N^{-2}+ C\delta  N^{r+1} e^{- \eta N^2 \frac{T}{2}})
\\ \nonumber
&\leq C e^{\eta T} \delta N^{r+3} + \frac{C e^{\eta T}}{N^2} + C N^{r+1} e^{- \eta N^2 \frac{T}{2}}
\\ \nonumber
&\leq \frac{C}{N} + \frac{C e^{\eta T}}{N^2} + C N^{r+1} e^{- \eta N^2 \frac{T}{2}} \ll 1,
\end{align}
where we have taken $N$ sufficiently large, depending on $T, \eta$ (in particular $N$ proportional to $\eta^{-1/2}$).
Under this choice of $\delta$ and $N$ we also have
$$
\| \delta^{-1}  \frac{M}{N} e^{-\eta (N^2 - 1) T} V_N \|_{H^r} \leq 
C N^{-r-5}e^{-\eta (N^2 - 1) T} \ll 1,
$$
that concludes the proof.
\end{proof}

We conclude the section with some remarks

\begin{rem}
\label{Remark:LargeDataWithStructure2D1}
The choice $u_0=0$ simplifies the proof but we may easily generalize the argument to small velocities, 
namely taking $\| u_0 \|_{H^r} = \varepsilon$ (for a sufficiently large $r$), 
where the size of the small parameter $\e$ depends on all the relevant parameters we introduced in the proof. As in the 3D case (see Remark \ref{Remark:LargeDataWithStructure1}) we can consider large data $u_0$ (and $\varepsilon$-perturbations of them)  introducing some extra structure. Moreover, in the 2D case we can also consider 
\emph{generic} large velocities as long as we take the viscosity $\nu$ sufficiently large. This result is proved in the next section. 
\end{rem}
\begin{rem}\label{Remark:LargeDataWithStructure2D2}
The result requires $\eta>0$ and we cannot promote it to the zero resistivity limit $\eta \to 0$; see Remark 
\ref{Remark:LargeDataWithStructure2}. 
\end{rem}
\begin{rem}
 The estimates do not blow up in the vanishing viscosity limit $\nu\to 0$ and one could in principle prove 
a reconnection statement for $\nu =0$. 
\end{rem}
\begin{rem}\label{Rem:FailsSStab}
As explained in the introduction, we can obtain a 3D reconnection result 
by the 2D result, extending the 2D magnetic field to a 3D object in the natural way (see \eqref{AsInHere}). 
However, the
3D reconnection result obtained in this way would not be {\it structurally stable} (in the sense of Remark \ref{Rem:StructStable}),
since we would need some structural stability of (lines of) degenerate
critical points (under 3D perturbations), which is in general false. Indeed, we 
can immediately recognize that the argument used 
in Lemma \ref{Lemma:CriticalPoints} fails for degenerate critical points.  
\end{rem}

\section{Magnetic reconnection in 2D (large velocity)}\label{Sec:Proof2DLarge}

We are now ready to show how to prove magnetic reconnection for a generic velocity in $H^4(\T^2)$ but taking the 
viscosity $\nu$ very large.

\begin{proof}[Proof of Theorem \ref{thm:main2}]
We consider initial data $(u_0, b_0)$ with $u_0 \in H^4(\T^2)$ and we consider the initial data
$$
b_0 = \frac{M}{N} V_N + \delta  V_1,
$$
where $M >0$ (is possibly large) and $\delta >0$ will be small (only depending on $N,M$) and $N$ large (depending on $\eta, T$) so that we have, as in the proof above, that $b_0$ has at least $4nm$ (regular) critical points.  More precisely we fix (recall \eqref{VeryExplCondIP})
\begin{equation}\label{FixDEltaT=0}
\delta = c \frac{M}{N^{L+1}},
\end{equation}
for some small constant $c>0$.

Let then $T>0$. Since (see again the previous section)
$$
e^{\eta T\Delta}\left(\frac{M}{N} V_N + \delta  V_1\right) = \frac{M}{N} e^{-\eta N^2 t} V_N + \delta e^{-\eta t} V_1,
$$
we can represent  $b(\cdot, T)$ using the Duhamel formula 
\begin{equation}
b(\cdot, T)=
\frac{M}{N} e^{-\eta N^2 T} V_N + \delta e^{-\eta T} V_1+ D(t) .
\end{equation}
where now
$$
D(t) := \int_0^T e^{\eta (T-s)\Delta}\dive\left(b(s,\cdot)\otimes u(s,\cdot)-u(s,\cdot)\otimes b(s,\cdot)\right) \de s,
$$
Thus we rescale $e^{\eta T}\delta^{-1}b(T,\cdot)$ and we must prove 
$$
\|  e^{\eta T}\delta^{-1}b(T,\cdot) - V_1 \|_{H^r} \ll 1,
$$
for a sufficiently large $r$, so that by Sobolev embedding and the structural stability of $V_1$ under $C^1$
perturbations (see Section \ref{Sec:BeltramiTaylorBis}) we know that $b(T,\cdot)$ (as the rescaled field) has only $4$ (regular) critical points. Since 
$$
\| e^{\eta T}\delta^{-1}b(T,\cdot) - V_1 \|_{H^r} =\|
e^{\eta T}\delta^{-1} \frac{M}{N} e^{-\eta N^2 T} V_N + e^{\eta T}\delta^{-1} D(t) \|_{H^r}
$$
we only need to bound the Sobolev norms of these two terms in a suitable way. If 
$$
R := \| u_0 \|_{H^{4}},
$$
then by \eqref{DecLargeNu} we have, after taking $\nu$ sufficiently large compared to $T$
$$
\|u(\cdot, t)\|_{H^4} \leq e^{\frac{C}{\eta^2}(R^2 + N^2 +1)} R^2 \frac{e^{-\eta t}}{\nu \eta}. 
$$
Thus
\begin{align}\nonumber
\| e^{\eta T}\delta^{-1} D(t) \|_{H^3} & \leq
C e^{\eta T}\delta^{-1}  
\int_0^T e^{-\eta (T-s)}\|b(s,\cdot)\|_{H^{4}}\|u(s,\cdot)\|_{H^{4}} \de s
\\ \nonumber
& \leq C \delta^{-1}  
\int_0^T e^{\eta s} \|b(s,\cdot)\|_{H^{4}}\|u(s,\cdot)\|_{H^{4}} \de s
\\ \nonumber
& \leq \frac{e^{C(R^2 + N^2 +1)} R^2}{\delta \nu \eta} 
\int_0^T  \|b(s,\cdot)\|_{H^{4}} \de s 
\leq \frac{e^{C(R^2 + N^2 +1)} R^2N^{2r}}{\delta \nu \eta \sigma} 
\ll 1
\end{align}
where we used \eqref{est:HrWD} in the penultimate inequality and we have then taken again $\nu$ sufficiently large (compared to all the fixed parameters).

Also we have 
$$
\left\| e^{\eta T} \delta^{-1} \frac{M}{N} e^{-\eta N^2 T} V_N  \right\|_{H^3}
\leq C \delta^{-1} MN^2 e^{-\eta  T (N^2 - 1)} \ll 1,
$$
where, recalling the definition \eqref{FixDEltaT=0} of $\delta$, the last inequality follows choosing $N$ sufficiently large. This concludes the proof.
\end{proof}

\begin{rem}
The assumption $u_0 \in H^4(\T^2)$ can be improved, however our goal was simply to provide examples 
in which we can consider large velocities at the initial time $t=0$.
\end{rem}

Remarks \ref{Remark:LargeDataWithStructure2D2} and \ref{Rem:FailsSStab} applies here too.

\section{Instantaneous reconnection}\label{Sec:Instant}

The aim of this section is to exhibit an example of instantaneous magnetic reconnection, proving 
Theorem \ref{thm:main3}. The 3D argument is exactly the same used in \cite{ELP} to prove instantaneous vortex reconnection 
for Navier--Stokes, so we will only sketch it for the sake of completeness. The 2D case requires some new ideas, but
it is in some extent simpler. Indeed, while the 3D argument relies upon the subharmonic Melnikov theory developed by  Guckenheimer and Holmes in \cite{GH}, in 2D one can use the more classical homoclinic/heteroclinic Melnikov method, in the heteroclinic formulation of Bertozzi \cite[pag. 1276]{B}. Moreover, a significant simplification of the argument is possible, based on the stream function formulation (we are grateful to Daniel Peralta-Salas for this observation), and we present it below. 
 
\begin{proof}[Proof of Theorem \ref{thm:main3}, d=3] 

We consider initial data $(u_0, b_0)$
with $b_0$ given by
\begin{equation}\label{TaylorT=0}
b_0 := M \left( \sin(x_3 + \varepsilon h), \cos(x_3 + \varepsilon h), - \varepsilon (\partial_{x_1} h)
\sin(x_3 + \varepsilon h) - ( \varepsilon \partial_{x_2} h ) \cos(x_3 + \varepsilon h)  \right) 
\end{equation}
where $h = h(x_1, x_2)$ is a 2D periodic function and $\varepsilon$ is a small parameter, 
both to be fixed later.
Note that $b_0$ is zero-average and divergence-free. Since $u_0$ may be large, we only have local solutions. 
However, this is acceptable for our purposes since we aim to prove instantaneous reconnection of some of the 
magnetic lines of $b_0$. Noting that $b_0$ is the pullback of the Beltrami field
$$
M \left( \sin x_3, \cos x_3 , 0 \right)
$$ 
under the volume preserving diffeomorphism
\begin{equation}\label{AgainB}
\Phi (x) =  (x_1, x_2, x_3 + \varepsilon h(x)),
\end{equation}
 we see that the integral lines of $b_0$ are periodic or quasi-periodic 
 and they form families of invariant  tori, like that one of \eqref{AgainB}. 
 We expand in Taylor series finding for small times $t > 0$:  
 \begin{equation}\label{FirstTaylorInTime}
 B(t) = b_0 + t( \eta \Delta b_0 - (u_0 \cdot \nabla ) b_0 + (b_0 \cdot \nabla) u_0 ) + \mathcal{O} (t^2).
 \end{equation}
 The goal is then to show that the Melnikov function associated to one of the periodic orbit (living on one of the resonant tori)
has a non-simple zero. The computation of the Melnikov function only depends on the linear part of the PDE above, 
since in the non resistive case $\eta =0$ the magnetic lines are transported, thus the Melnikov function must be identically 
zero. Once we have noted this fact, we can proceed exactly as in the proof of Theorem 1.4 of \cite{ELP} to show that, choosing
$h(x_1, x_2) = \cos(p X_1 - q x_2), $
where $p/q$ is rational, for all $\varepsilon >0$ sufficiently small and
for all sufficiently small times $t > 0$ the resonant torus corresponding to 
$\cot X_3 = p/q$ is instantaneously broken.


\end{proof}

\begin{rem}
By some standard PDEs considerations and invoking a slightly more general version of the stability Theorem \ref{thm:stab}
(which takes into account small initial velocities rather than zero ones) one can show that if we restrict to
small (say) $H^1$ velocities the (strong) solutions constructed in the previous proof 
are indeed {\it global}.
\end{rem}

We conclude the paper with the proof of the 2D instantaneous reconnection result.

\begin{proof}[Proof of Theorem \ref{thm:main3}, d=2] 
Let us consider the following vector field
\begin{equation}\label{eq:perturbed_initdat}
b_0^\e= M(-\sin(x_1)\sin(x_2 - \e x_1 ), - \e \sin(x_1)\sin(x_2 - \e x_1 )-\cos(x_1)\cos(x_2 - \e x_1).
\end{equation}
Note that $\dive b_0^\e=0$. Note also 
that for $\e=0$ the vector field $b_0^\e$ corresponds to $-\mathcal V_{11}^4$ and we denote it by $b_0$. The integral lines of $b_0$ are the solutions of 
\begin{equation}\label{eq:ode-unp}
\begin{cases}
\dot{x_1}=- M\sin(x_1)\sin(x_2),\\
\dot{x_2}=- M\cos(x_1)\cos(x_2),
\end{cases}
\end{equation}
where the dot denotes the derivative with respect to the parametrization. 
The integral lines (of the system above) presents heteroclinic orbits connecting four saddle points. 
This is a peculiarity shared by the phase diagram of the Taylor fields. 
Note that, if we would consider $u_0=0$, the solution of $\eqref{eq:mhd}$ would be 
given by $(u,b)=(0,-Me^{-2t}b_0(x))$: obviously the magnetic field $b$ is topologically equivalent to 
$b_0$ (that is $-\mathcal V_{11}^4$ from Section \ref{Subsec:TaylorV}) and there would be no reconnection at any time.
However, if we perturb $b_0$ as in \eqref{eq:perturbed_initdat} we are able to 
provide an example of instantaneous change of the topology of the integral lines. 

We start noting that
$b_0^\e$ is given by the pull-back of $b_0$
$$
b_0^\e=\Phi^*b_0,
$$
under 
 the change of variables
\begin{equation}\label{def:phi}
\Phi(x_1,x_2)=(x_1,x_2  - \e x_1).
\end{equation}
It is clear that $\Phi$ is a volume preserving diffeomorphism, 
thus, in particular, the integral lines of $b_0^\e$ are topologically equivalent to that 
of~$b_0$. However, $b_0^\e$ is not a Taylor field (in particular it is not an eigenvector of the Laplacian), and its structure can change during the evolution of the fluid. 

We consider a smooth vector field 
$u_0$ with zero divergence. The system \eqref{eq:mhd} with initial datum $(u_0,b_0^\e)$ admits a
global smooth solution that we denote by $(u^\e,b^\e)$. 
By using the equation, we Taylor expand $b^\e$ with respect to time obtaining that
\begin{equation}\label{TaylorInTimeSol}
b^\e(t,x)=b_0^\e+t(\eta\Delta b_0^\e+(b_0^\e\cdot\nabla) u_0-(u_0\cdot\nabla)b_0^\e) + \mathcal{O}(t^2),
\end{equation}
here $t$ has to be seen as the perturbative parameter. Of course, this representation holds for short times, that is 
a harmless restriction since we want to prove an instantaneous reconnection result.

Since $\dive b^\e(t, \cdot)=0$ (and the average on $\T^2$ is zero as well)
there exists a stream function (or Hamiltonian) $\psi^\e(t,  \cdot)$, namely a scalar function such that
$$
b^\e(t, x)=\nabla^\perp \psi^\e(t, x).
$$
We denote with $\psi_0^\e$ the stream function at time $t=0$ and recalling the definition \eqref{eq:perturbed_initdat} of $b_0^\e$ it is easy to verify that 
$$
\psi_0^\e(x_1,x_2)=-\sin x_1 \cos(x_2-\e x_1).
$$
To fix the ideas, we consider the saddle connection given by 
$$
x_2=\e x_1+\frac{\pi}{2},
$$
with $x_1\in[0,\pi]$. This is an heteroclinic orbit connecting the 
saddle points $A=(0,\frac{\pi}{2})$ and $B=(\pi, \e\pi+\frac{\pi}{2})$. 
Thus for the stream function we have 
$\psi_0^\e(A)=\psi_0^\e(B)$.

Taking this into account, it is
easy to show that, since the Hamiltonian is constant along heteroclinic orbits, the connection is (instantly) broken if we can prove that
\begin{equation}\label{PartialInt=0}
(\partial_{t} \psi^\e)(0, A) 
\neq 
(\partial_{t} \psi^\e)(0, B).
\end{equation}

Thus it remains to verify this condition. 
From \eqref{TaylorInTimeSol} and the second equation in \eqref{eq:mhd}, we deduce that $\psi^\e$ verifies 
\begin{equation}\label{TaylorInTimePsi}
\psi^\e(t,x)=\psi_0^\e +t( \eta \Delta \psi_0^\e + (u_0 \cdot \nabla) \psi_0^\e )+ \mathcal{O}(t^2).
\end{equation}
In order to prove \eqref{PartialInt=0} we only need to take into account the contribution of 
$t \eta \Delta \psi_0^\e$ to it. Indeed, the contributions of the term
$ t (u_0 \cdot \nabla) \psi_0^\e$ 
cancel out 
because they exactly correspond to what we would have in the case $\eta=0$, when we know that the
magnetic lines are transported by the fluid (Alfven's theorem), so that we would have an 
equality in (the analogous of) \eqref{PartialInt=0}. 
Then, by a direct computation we obtain that
$$
\Delta\psi_0^\e(A)=-2\e
$$
$$
\Delta\psi_0^\e(B)=
2\e \cos (\e \pi),
$$
and in particular $\Delta\psi_0^\e(A)\neq\Delta\psi_0^\e(B)$ for $\e \in (0, 1)$,
that concludes the proof. 

\end{proof}

\begin{rem}
The previous arguments cannot be promoted to the zero resistivity limit $\eta \to 0$, since they work in 
a regime where $\eta \gg t$. Indeed otherwise we could not consider the $\mathcal{O}(t^2)$ terms 
in~\eqref{FirstTaylorInTime}-\eqref{TaylorInTimePsi} as higher order perturbations of the leading 
terms $t(\eta\Delta b_0^\e+(b_0^\e\cdot\nabla) u_0-(u_0\cdot\nabla)b_0^\e)$ and $t( \eta \Delta \psi_0^\e + (u_0 \cdot \nabla) \psi_0^\e )$. Thus, as~$\eta \to 0$,
also the range of times $t >0$ for which we have proved a change of topology (w.r.t. the initial time~$t=0$)
shrinks to zero. 
 \end{rem}


\section*{Acknowledgements}
The authors are grateful to Daniel Peralta-Salas for his comments and for suggesting how to simplify the proof of the Theorem \ref{thm:main3} in dimension $2$. This research has been supported by the Basque Government under program BCAM- BERC 2022-2025 and by the Spanish Ministry of Science, Innovation and Universities under the BCAM Severo Ochoa accreditation SEV-2017-0718 and by the projects PGC2018-094528-B-I00, PID2021-123034NB-I00 and PID2021-122156NB-I00. GC is also supported by the ERC Starting Grant 101039762 HamDyWWa. RL is also supported by the Ramon y Cajal fellowship RYC2021-031981-I.

\end{document}